\documentclass[11pt]{article}
\usepackage{color,xcolor}
\usepackage{graphicx,mathrsfs}
\usepackage{amssymb,bbm}
\usepackage{amsthm}
\usepackage{amsmath}
\usepackage[margin=1in]{geometry}
\usepackage{caption,numbysec}
\usepackage{subcaption}
\usepackage{float}
\usepackage{mdframed,pdfsync,mathrsfs} 
\usepackage{esint}
\usepackage{dsfont}
\usepackage[normalem]{ulem}
\usepackage{verbatim}
\allowdisplaybreaks

\newtheorem{assumption}{Assumption}
\usepackage[colorlinks=true,pdfpagemode=UseNone,urlcolor=blue,linkcolor=blue,citecolor=blue,breaklinks=true]{hyperref}

\usepackage{cleveref}

\usepackage{enumitem}

\newtheorem{thm}{Theorem}[section]
\newtheorem{lemma}[theorem]{Lemma}
\newtheorem{proposition}[theorem]{Proposition}

\newtheorem{remark}[theorem]{Remark}

\newcommand{\bea}{\begin{eqnarray*}}
\newcommand{\eea}{\end{eqnarray*}}
\newcommand{\ben}{\begin{eqnarray}}
\newcommand{\een}{\end{eqnarray}}
\newcommand{\beq}{\begin{equation}}
\newcommand{\eeq}{\end{equation}}

\newcommand{\R}{\ensuremath{\mathbb{R}}}

\newcommand{\1}{\mathds{1}}
\newcommand{\bangle}[1]{\langle #1\rangle}
\newcommand{\muinf}{\mu_\infty}
\newcommand{\nuinf}{\nu_\infty}
\newcommand{\dotH}{\dot{H}}

\newcommand{\cX}{\mathcal{X}}
\newcommand{\cY}{\mathcal{Y}}

\newcommand{\cwp}{c_{\mathrm{WP}}}

\newcommand{\bal}{\begin{aligned}}
\newcommand{\enbal}{\end{aligned}}
\newcommand{\be}{\begin{equation}}
\newcommand{\ee}{\end{equation}}

\newcommand{\eps}{\varepsilon}

\renewcommand{\d}{\partial}

\newcommand{\dd}{\,\mathrm{d}}

\newcommand{\calR}{\mathcal{R}}

\begin{document}
\title{Local Algebraic Stability of Wasserstein Gradient Descent-Ascent and Mean Field Langevin Dynamics without Spectral Gap}
\author{Jing An\thanks{School of Mathematical Sciences, MOE-LSC, Shanghai Jiao Tong University, Shanghai, 200240, China (jingan1009@sjtu.edu.cn)}  \and Lihan Wang\thanks{ Department of Mathematics, National University of Singapore, 119076, Singapore (lihanw@nus.edu.sg)}}
\date{}
\maketitle
\numberbysection

\begin{abstract}
    We study the local stability of Wasserstein gradient descent-ascent (GDA) for entropy regularized two-player zero-sum games on noncompact domains. When one or both equilibrium measures fail to satisfy a Poincar\'e inequality, the spectral gap arguments underlying local exponential convergence are unavailable. we show that Wasserstein GDA converges algebraically to the mixed Nash equilibrium in chi-squared divergence, provided the initial relative density is a small perturbation in $L^2$ with bounded algebraic moments. The proof has two main ingredients: a cancellation in a negative Sobolev energy estimate, and a bootstrap argument that keeps the moments bounded uniformly in time. For polynomial-tailed equilibria, we construct perturbations supported far out in the tails that survive over diffusive time scales, which gives algebraic lower bounds on the worst-case relaxation rate. We also establish analogous local algebraic stability and lower bound results for mean-field Langevin dynamics, when the equilibrium fails to satisfy a spectral gap. 
\end{abstract}
\section{Introduction}
In this paper, we consider the long-time behavior of Wasserstein gradient descent-ascent (GDA) for entropy-regularized two-player zero-sum games \cite{hsieh2019finding, domingo2020mean} on noncompact domains. More precisely, given a game objective function $k$, the two players act simultaneously, one minimizing and the other maximizing $k$, with the goal of reaching a Nash equilibrium (or saddle point) of the min-max problem \[\min_{x\in \cX}\min_{y\in \cY} k(x,y).\]
Such games arise prominently in many areas of machine learning, including generative adversarial networks \cite{goodfellow2014generative}, multi-agent reinforcement learning \cite{busoniu2008comprehensive}, and distributionally robust learning \cite{sinha2017certifying, madry2017towards}. However, if one naively performs gradient descent in $x\in\cX$ and gradient ascent in $y\in\cY$, the Nash equilibrium may be difficult or even impossible to reach. This work follows the framework of \cite{domingo2020mean}, in which the players choose mixed strategies, modeled as probability measures rather than points, and an entropic regularization term is added. This leads to the following problem:
\begin{equation}\label{eq:MNE}
\begin{aligned}
    &\min_{\mu\in\mathcal{P}(\cX)} \max_{\nu\in\mathcal{P}(\cY)} E_{\beta} (\mu, \nu): =E(\mu, \nu) + \beta^{-1} H(\mu)- \beta^{-1} H(\nu)\\
    &\text{where}\quad E (\mu, \nu):=\int_{\cY} \int_{\cX} k(x,y) \dd\mu(x) \dd\nu(y) , \quad \text{and}~H(\mu):=\int_{\cX}\log \frac{\dd\mu}{\dd x} ~\dd\mu.
\end{aligned}
\end{equation}
In this framework, the saddle point, which we denote as $(\muinf,\nuinf)$ for the rest of this work, also known as a mixed Nash equilibrium (MNE), usually exists; in compact space, this is a consequence of Glicksberg's Theorem \cite{glicksberg1952further}, and is proved in \cite{conforti2023game, conger2024coupled} for general state space. In addition, by the optimality condition of $E_\beta$, we learn that the MNE should satisfy  
\begin{equation}\label{eqn:equilibrium}
    \log \mu_{\infty} = -\beta \int_{\cY}  k(x,y) \dd\nu_{\infty}(y)+\text{const},\quad \log \nu_{\infty} = \beta \int_{\cX}  k(x,y) \dd\mu_{\infty}(x)+\text{const}.
\end{equation}
To find the MNE of this regularized problem, it is natural to consider the gradient descent-ascent flow (GDA) under the Wasserstein metric, which takes the form
 \begin{equation}\label{eqn:GDA}
     \begin{aligned}
         &\d_t \mu_t = \nabla_x \cdot\Big(\mu_t \int_{\cY} \nabla_x k(x,y) \dd\nu_t(y)\Big)+ \beta^{-1}\Delta_x \mu_t\\
         &\d_t \nu_t = -\nabla_y \cdot\Big(\nu_t \int_{\cX} \nabla_yk(x,y) \dd\mu_t(x)   \Big)+ \beta^{-1}\Delta_y \nu_t.
     \end{aligned}
 \end{equation}
For the well-posedness of \eqref{eqn:GDA}, see \cite{domingo2020mean, conforti2023game} for probabilistic approaches, \cite{isobe2025gradient} for the evolution variational inequality formulation, and \cite[Appendix A.3]{seo2026local}.
 
 Wasserstein GDA (\ref{eqn:GDA}) has attracted considerable attention, and the central question is its long-time convergence behavior  \cite{wang2024open}; see \Cref{subsec:references} for a list of references. Local exponential convergence to the MNE, meaning that $(\mu_t,\nu_t)$ converges to $(\muinf,\nuinf)$ whenever the initial condition $(\mu_0,\nu_0)$ is already close to $(\muinf,\nuinf)$, was recently established in \cite{seo2026local, wang2026local}. The result of \cite{seo2026local} requires $\cX$ and $\cY$ to be compact, and gives exponential convergence in Wasserstein distance by exploiting the convexity-concavity of $E(\mu,\nu)$ near the MNE along Wasserstein geodesics. The result of \cite{wang2026local} is formulated in chi-squared divergence and essentially relies on a Poincar\'e inequality for both $\muinf$ and $\nuinf$. Both works embody the principle that strong coercivity of the local convex-concave structure at equilibrium yields local exponential stability for Wasserstein GDA. Global convergence, by contrast, can fail: \cite{mourrat2026oscillating} constructs an oscillatory solution in low temperature, following the companion result \cite{mourrat2026pl} in finite dimensions.

 However, the Poincar\'e inequality assumed in \cite{wang2026local} requires exponential integrability of both equilibria $\muinf$ and $\nuinf$ \cite{bakry_analysis_2014}, which may not be satisfied in practice, especially when the state spaces $\cX,\cY$ are unbounded. This motivates the following question: does the MNE remain asymptotically stable when one or both equilibrium measures fail to satisfy a Poincar\'e inequality, and what relaxation rates can then be expected? This roughly corresponds to a weak (as opposed to strong) coercivity of the energy landscape. The main assumption of our work is the so-called \emph{weighted Poincar\'e inequality} (see \eqref{eqn:cwpx}, \eqref{eqn:cwpy}), which quantitatively characterizes the aforementioned weak coercivity, and has been studied in \cite{Muckenhoupt1972, blanchet2007hardy, bobkov2009weighted}. Weighted Poincar\'e inequality can be used to study the algebraic relaxation of the Fokker--Planck equation, one of the simplest Wasserstein gradient flows, to its equilibrium, see \cite[Appendix B.2]{bouin2021hypocoercivity}. 

In \Cref{sec:wPIconv}, we establish local algebraic stability under weighted Poincar\'e inequalities and quantitative tail assumptions on the MNE. More precisely, let $f_t:=\frac{\dd\mu_t}{\dd\muinf}-1$ and $g_t:=\frac{\dd\nu_t}{\dd\nuinf}-1$ denote the relative density perturbations. If $(f_0,g_0)$ is small in $L^2$ and has bounded algebraic moments, we prove that $(f_t,g_t)$ converges to zero in $L^2$ at an algebraic rate, so that $(\mu_t,\nu_t)\to(\muinf,\nuinf)$. The proof combines a crucial cancellation in the $H^{-1}$ energy estimate observed in \cite{wang2026local} with weighted interpolation and a nonlinear bootstrap on moment bounds, which is also necessary for the Fokker--Planck equation \cite{bouin2021hypocoercivity}.

In \Cref{sec:lower-bound}, we prove lower bounds for polynomial-tailed equilibria. We construct localized, mean-zero perturbations far out in the tail and show that they relax slowly under the full coupled nonlinear evolution. For every sufficiently large observation time, there is an admissible perturbation that satisfies a quantitative algebraic lower bound over a comparable time interval. Hence the interaction between the two players does not eliminate the slow modes caused by weak confinement, and exponential convergence cannot hold uniformly over the class of admissible initial data.

Closely related to Wasserstein GDA is the mean-field Langevin dynamics (MFLD) \cite{hu2021mean}, given by
 \begin{equation}\label{eqn:MFLDorig}
     \partial_t \mu_t = \nabla_x \cdot (\mu_t \nabla F'(\mu_t))+\beta^{-1}\Delta_x \mu_t,
 \end{equation}
where $F'$ denotes the first variation of $F$. It is the Wasserstein gradient flow \cite{ambrosio2008gradient} of the entropy-regularized functional
 \begin{equation}\label{eqn:MFLDFtau}
     F_\tau(\mu)=F(\mu)+\beta^{-1} H( \mu).
 \end{equation}
 The local convergence behavior is also studied in \cite{wang2026local}, using mostly similar techniques as Wasserstein GDA. In \Cref{sec:MFLD}, we will follow the similar setting as in \cite{wang2026local} and study the local convergence behavior when Poincar\'e inequality is no longer satisfied by the equilibrium distribution.

%\textcolor{red}{JA: is this paragraph okay?} 
Lastly, for completeness, in Appendix \ref{sec:hyp} we explore the hypocoercive structure of the two-timescale Wasserstein GDA, which is an extension of \cite{an2025convergence} in analyzing hypocoercivity in finite dimensions. We construct modified Lyapunov energy in the $\dot{H}^{-1}$ norm in the different regimes of the learning-rate ratio, and control the resulting error terms using the stability results of \Cref{sec:wPIconv}. This shows how the skew-symmetric part of the linearized dynamics interacts with its symmetric part, and yields local algebraic convergence results that hold uniformly across all learning-rate ratios.

\subsection{Main Results}\label{subsec:theorems}
For clarity, we assume $\cX=\R^{d_x}, \, \cY=\R^{d_y}$, although our results can be generalized to unbounded manifolds with suitable modifications. We now state the assumptions on the kernel $k$ and, correspondingly, on the MNE $(\muinf,\nuinf)$.
\begin{assumption}\label{ass:wpi}
    The entropy regularized min-max problem \eqref{eq:MNE} has a unique MNE $(\muinf,\nuinf)$. Moreover, the kernel $k$ and the corresponding equilibrium distributions $\mu_\infty, \, \nu_\infty$ satisfy the following:
    \begin{itemize}
        \item The kernel $k\in C^2(\R^{d_x}\times \R^{d_y})$, and 
        \begin{equation} \label{eqn:bddk2}
            \|\nabla^2 k(x,y)\|_{L^\infty(\R^{d_x}\times \R^{d_y})} = M<\infty.
        \end{equation}
   %     By \eqref{eqn:equilibrium}, we have 
    %    \begin{equation}
     %       \|\nabla^2_x \log \mu_\infty(x)\|, \,  \|\nabla^2_y \log \nu_\infty(y)\|\le M.
      %  \end{equation}
       % We can therefore assume that there exists a constant $b\in \R$, such that %\textcolor{red}{Do we need assumption (2.7)? I guess (2.11) is strong enough}
     %   \begin{equation}
    %        \|\nabla_x \log \mu_\infty(x)\|\le M\bangle{x}+b, \qquad  \|\nabla_y \log \nu_\infty(y)\|\le M\bangle{y}+b.
   %     \end{equation}
         \item Weighted Poincar\'e inequality: there exist constants $\cwp^x,\, \cwp^y>0$ and exponents $\sigma_x,\, \sigma_y\in(0,1] $ such that for any $f,g$ with $\int f \dd \muinf(x)=0$ and $\int g\dd \nuinf(y)=0$,
        \begin{equation}\label{eqn:cwpx}
            \int \frac{f(x)^2}{\bangle{x}^{2\sigma_x}} \dd \mu_\infty(x) \le \frac{1}{\cwp^x}  \int|\nabla_x f(x)|^2\dd \mu_\infty(x),
        \end{equation}
        and 
         \begin{equation}\label{eqn:cwpy}
            \int \frac{g(y)^2}{\bangle{y}^{2\sigma_y}}\dd \nu_\infty(y) \le \frac{1}{\cwp^y}  \int|\nabla_y g(y)|^2\dd \nu_\infty(y).
        \end{equation}
      Here $\bangle{x}:=\sqrt{1+|x|^2}$, and we set $\sigma_* :=\max\{\sigma_x,\sigma_y\}$.
          \item Decay at infinity: there exists some $\alpha> 0$  such that 
          \begin{equation} \label{eqn:bddmomentorig}
            \int  \bangle{x}^\alpha \dd\mu_\infty(x) + \int  \bangle{y}^\alpha \dd\nu_\infty(y)<\infty. 
        \end{equation}
          In addition, there exist constants $a\ge 0, r_0\ge 0, p>\max\{d_x, d_y\}+\alpha$ such that 
        \begin{equation}\label{eqn:tailequil}
            x\cdot \nabla \log \muinf(x) \le a \1_{\{|x|\le r_0\}}-p, \qquad y\cdot \nabla \log \nuinf(y) \le a\1_{\{|y|\le r_0\}}-p.
        \end{equation}
    \end{itemize}
\end{assumption}

\begin{remark}
 % The assumptions \eqref{eqn:bddmomentorig} and \eqref{eqn:tailequil} are motivated by the benchmark example $\muinf(x) \propto \bangle{x}^{-p}$. For this example, \eqref{eqn:tailequil} holds with $a=0, \, r_0=1$, and $p>\max\{d_x,d_y\}+\alpha$ is required for \eqref{eqn:bddmomentorig} to hold. They also hold if $\muinf$ or $\nuinf$ decays faster than any algebraic function at infinity.
 Assumptions \eqref{eqn:bddmomentorig} and \eqref{eqn:tailequil} are motivated by the benchmark example $\muinf(x)\propto\bangle{x}^{-q}$ with $q>d_x+\alpha$. The condition $q>d_x+\alpha$ is exactly what makes \eqref{eqn:bddmomentorig} hold. For \eqref{eqn:tailequil}, a direct computation gives
\[
x\cdot\nabla\log\muinf(x)\propto-\frac{q|x|^2}{1+|x|^2}.
\]
Hence, for any $p\in(d_x+\alpha,\,q)$, \eqref{eqn:tailequil} holds with $a\ge 0$ and $r_0^2=p/(q-p)$. The same applies to $\nuinf$ with $d_y$ in place of $d_x$. Both assumptions also hold when $\muinf$ or $\nuinf$ decays faster than any algebraic function at infinity, for instance, Gaussian-type tails.
\end{remark}
\begin{remark}\label{rmk:easewpi}
The weighted Poincar\'e inequality \eqref{eqn:cwpx} holds for a broad class of densities. Two representative examples are:
\begin{itemize}
    \item If $\muinf(x)\sim\exp(-c\bangle{x}^{\theta})$ as $|x|\to\infty$ for some $c>0$ and $\theta\in(0,1)$, then \eqref{eqn:cwpx} holds with $\sigma_x=1-\theta$ (see~\cite{bouin2021hypocoercivity}), and \eqref{eqn:bddmomentorig} holds for every $\alpha\in(0,\infty)$.
    \item If $\muinf(x)\sim\bangle{x}^{-(d_x+q)}$ as $|x|\to\infty$ for some $q>0$, then \eqref{eqn:cwpx} holds with $\sigma_x=1$ (see~\cite{blanchet2007hardy, bobkov2009weighted}), and \eqref{eqn:bddmomentorig} holds for every $\alpha\in(0,q)$.
\end{itemize}
    The same statements hold for $\nuinf$ with $d_y$ and $\sigma_y$ in place of $d_x$ and $\sigma_x$. It is therefore easy to construct equilibrium distributions that satisfy Assumption~\ref{ass:wpi}. Let $k(x,y)=\beta^{-1}\big(U(x)-V(y)\big)$, where $U\in C^2(\R^{d_x})$ and $V\in C^2(\R^{d_y})$ have bounded second derivatives and satisfy
\[
U(x)= a_x\log\bangle{x}\quad (|x|\gg 1),\qquad V(y)= a_y\log\bangle{y}\quad (|y|\gg 1),
\]
for constants $a_x,a_y>0$ to be chosen. By \eqref{eqn:equilibrium}, the equilibrium is
\[
\muinf(x)\propto e^{-U(x)},\qquad \nuinf(y)\propto e^{-V(y)},
\]
so that $\muinf(x)\propto\bangle{x}^{-a_x}$ and $\nuinf(y)\propto\bangle{y}^{-a_y}$ for $|x|\gg 1$ and $|y|\gg 1$. Since $k$ is separable, this MNE is unique. If $a_x,a_y>\max\{d_x,d_y\}+\alpha$, then Assumption~\ref{ass:wpi} is satisfied, with $\sigma_x=\sigma_y=1$. In this case, neither $\muinf$ nor $\nuinf$ satisfies a standard Poincar\'e inequality, because polynomially decaying densities are not exponentially integrable.
\end{remark}
\begin{remark}
    Compared to \eqref{eqn:bddk2}, \cite{wang2026local} only assumes $\|\nabla_x\nabla_y k(x,y)\|_{L^\infty(\R^{d_x}\times \R^{d_y})} <\infty$. A careful check of our paper finds that all results in the main sections of our work hold under the weaker assumption. However, the strengthened assumption is required in the Appendix \ref{sec:hyp}. We therefore keep the stronger assumption for simplicity.
\end{remark}
We now present our first theorem, which establishes the local algebraic $\chi^2$-decay of $(\mu_t,\nu_t)$ to the MNE along the flow \eqref{eqn:GDA}.
\begin{thm}\label{thm:decay_upper_bd}
    Let $(\mu_t,\nu_t)$ be a solution of the Wasserstein GDA \eqref{eqn:GDA} with initial condition $(\mu_0,\nu_0)$, and define the relative density perturbations $f_t := \frac{\dd \mu_t}{\dd \muinf}-1$ and $g_t := \frac{\dd \nu_t}{\dd \nuinf}-1$. Suppose Assumption~\ref{ass:wpi} holds with $\alpha >4\sigma_*$, and fix an energy level $R>0$. Then there exist constants $\varepsilon>0$ and $C>0$, depending only on $\alpha,\, \sigma_*,\, R,\, \cwp^x,\, \cwp^y,\, \beta$, such that the following holds. If the initial perturbations satisfy
    \begin{equation} \label{eqn:bddmoment}
        \int f_0^2 \bangle{x}^\alpha \dd\muinf(x) +\int g_0^2 \bangle{y}^\alpha \dd\nuinf(y)<\frac{R}{2}, \qquad \text{and} \qquad \|f_0\|^2_{L^2(\muinf)}+\|g_0\|^2_{L^2(\nuinf)}\le \varepsilon,
    \end{equation}
    then the solution $(f_t,g_t)$ satisfies the uniform-in-time moment bound
    \begin{equation}\label{eqn:uniform-Ealpha-bound}
        \sup_{t\ge 0}\Big\{\int f_t^2 \bangle{x}^\alpha \dd\muinf(x) +\int g_t^2 \bangle{y}^\alpha \dd\nuinf(y)\Big\}\le R.
    \end{equation}
    Moreover, $(f_t,g_t) \to 0$ at an algebraic rate:
    \begin{equation}\label{eqn:At_upper_decay}
        \|f_t\|^2_{L^2(\muinf)}+\|g_t\|^2_{L^2(\nuinf)} \le C\big(Z_0^{-\frac{2\sigma_*}{\alpha-2\sigma_*}}+t\big)^{-\frac{\alpha-2\sigma_*}{2\sigma_*}}.
    \end{equation}
    Here $Z_0:= \|f_0\|^2_{\dotH^{-1}(\muinf)}+\|g_0\|^2_{\dotH^{-1}(\nuinf)}$ is finite by \eqref{eqn:bddmomentorig} together with the weighted Poincar\'e inequalities \eqref{eqn:cwpx} and \eqref{eqn:cwpy}.
\end{thm}
We briefly explain the strategy behind the proof of \Cref{thm:decay_upper_bd}; see also \cite[Appendix B]{bouin2021hypocoercivity}. As a toy model, suppose $\mu_t$ solves the Fokker--Planck equation, a benchmark example of a Wasserstein gradient flow \cite{jordan1998variational}:
\[
\partial_t \mu_t = \nabla_x\cdot\Big(\mu_t \nabla_x\log\frac{\mu_t}{\muinf}\Big).
\]
Let $f_t:=\frac{\dd\mu_t}{\dd\muinf}-1$, so that $\int f_t\,\dd\muinf=0$ for all $t$. If $\muinf$ satisfies a standard Poincar\'e inequality, then for any $f_0\in L^2(\muinf)$ the $L^2$ energy estimate directly gives exponential decay of $f_t$ in $L^2(\muinf)$. This is the same as exponential decay of $\chi^2(\mu_t\,\|\,\muinf)=\|f_t\|^2_{L^2(\muinf)}$, with the same rate.

If the Poincar\'e inequality fails for $\muinf$, we can still assume the weighted Poincar\'e inequality \eqref{eqn:cwpx}. In that case, however, $f_0\in L^2(\muinf)$ is no longer sufficient. Suppose in addition that we have a uniform-in-time moment bound
\begin{equation}\label{eqn:simplemomentest}
    \sup_{t\ge 0}\int  f_t^2 \bangle{x}^\alpha\,\dd\muinf(x)<\infty.
\end{equation}
Then H\"older's inequality gives
\[
\int f_t^2\,\dd\muinf \le \Big(\int \frac{f_t^2}{\bangle{x}^{2\sigma_x}}\,\dd\muinf\Big)^{\frac{\alpha}{\alpha+2\sigma_x}}\Big(\int f_t^2\bangle{x}^{\alpha}\,\dd\muinf\Big)^{\frac{2\sigma_x}{\alpha+2\sigma_x}}
\lesssim \Big(\int \frac{f_t^2}{\bangle{x}^{2\sigma_x}}\,\dd\muinf\Big)^{\frac{\alpha}{\alpha+2\sigma_x}}.
\]
Combining this with \eqref{eqn:cwpx}, the energy estimate becomes
\[
\frac{\dd}{\dd t}\frac12\int f_t^2\,\dd\muinf = -\int|\nabla_x f_t|^2\,\dd\muinf \stackrel{\eqref{eqn:cwpx}}{\lesssim} -\int\frac{f_t^2}{\bangle{x}^{2\sigma_x}}\,\dd\muinf \lesssim -\Big(\int f_t^2\,\dd\muinf\Big)^{\frac{\alpha+2\sigma_x}{\alpha}}.
\]
An ODE argument then yields $\|f_t\|^2_{L^2(\muinf)}\lesssim t^{-\frac{\alpha}{2\sigma_x}}$ as $t\to\infty$. For the Fokker--Planck equation, \eqref{eqn:simplemomentest} propagates in time once it holds initially. For the coupled GDA dynamics this is no longer automatic, so in the proof of \Cref{thm:decay_upper_bd} we add a bootstrap argument that handles the interplay between moment bounds such as \eqref{eqn:simplemomentest} and the $\dot H^{-1}$ and $L^2$ energy estimates of \cite{wang2026local}.

Appendix~\ref{sec:hyp} extends \Cref{thm:decay_upper_bd} to the two-timescale setting, with an emphasis on the underlying operator structure. Following the finite-dimensional hypocoercivity analysis of An and Lu~\cite{an2025convergence}, we split the rescaled Wasserstein GDA into a symmetric dissipative part and a skew-symmetric coupling in the homogeneous negative Sobolev space. Resolvent-based cross terms in a modified Lyapunov energy then show how the skew-symmetric operator transfers perturbations into directions dissipated by the symmetric operator. The nonlinear error is controlled by the stability result of \Cref{thm:decay_upper_bd}. All together, we obtain local algebraic convergence for all range of timescale ratios. We place this extension in the appendix because the hypocoercive formulation complements the main stability theorem, whereas the mechanism that produces algebraic decay without a spectral gap is already developed in \Cref{sec:wPIconv}.

To complement \Cref{thm:decay_upper_bd}, we also show that there exist systems for which algebraic decay is unavoidable. Assumption~\ref{ass:wpi} alone is not enough for this, since it also allows strongly confining equilibria, such as Gaussian measures, for which a spectral gap and exponential convergence may hold. We therefore impose a two-sided polynomial-tail assumption on the component with the largest weighted Poincar\'e exponent $\sigma_*$.
\begin{assumption}\label{ass:twosided}
 We identify
$\mu_\infty$ with its density with respect to Lebesgue measure and assume
that there exist constants $R_0>0$, $p>d_x+\alpha$, and
$0<c_\mu\leq C_\mu<\infty$ such that
\begin{equation}
    c_\mu\bangle{x}^{-p}
    \leq
    \mu_\infty(x)
    \leq
    C_\mu\bangle{x}^{-p},
    \qquad |x|\geq R_0.
    \label{eqn:two-sided-polynomial-tail}
\end{equation}
In this case $\sigma_*=1$; see \Cref{rmk:easewpi}. The argument is analogous if $\nuinf$ has the same type of decay \eqref{eqn:two-sided-polynomial-tail} at infinity; in that case, we interchange the roles of the $x$- and $y$-variables in what follows.
\end{assumption}
\begin{thm}\label{thm:lowerbd}
    Suppose Assumptions~\ref{ass:wpi} and~\ref{ass:twosided} hold. Then for every sufficiently large $T>0$ there exists an initial condition $h_0^{(T)}=(f_0^{(T)},g_0^{(T)})^\top$ with
    \[
        \int f_0^{(T)}\,\dd\muinf=0, \qquad \int g_0^{(T)}\,\dd\nuinf=0, \qquad 1+f_0^{(T)}\ge 0, \qquad 1+g_0^{(T)}\ge 0,
    \]
    which satisfies the hypotheses \eqref{eqn:bddmoment} of \Cref{thm:decay_upper_bd}. Moreover, there is a constant $c>0$, independent of $T$, such that the corresponding solution $h_t^{(T)}= (f_t^{(T)}, g_t^{(T)})^\top$ satisfies
    \begin{equation}\label{eqn:lower_bds}
        \|f_t^{(T)}\|^2_{L^2(\muinf)}+\|g_t^{(T)}\|^2_{L^2(\nuinf)} \ge c\,(1+t)^{-\frac{p-d_x}{2}} \qquad \text{for all } t\in[T/2,T].
    \end{equation}
\end{thm}
%\textbf{Motivation:} In the whole space $\Rm^{d_x}\times\Rm^{d_y}$,  how to obtain a unified convergence of the gradient descent-ascent flow for all range of $\eta$ under the weighted Poincar\'e inequality assumption?
Analogous local long-time stability and lower bound results can also be established for MFLD, which are the remaining main theorems of this work. The following assumption is largely borrowed from \cite[Assumption A]{wang2026local}, with the obvious weak confinement modifications, identical to those in Assumption~\ref{ass:wpi}.
\begin{assumption}\label{ass:MFLD}
    The functional $F: \mathcal{P}_2(\R^d)\rightarrow \R$ is displacement-smooth, and there exists a
stationary measure $\muinf$ of the MFLD, that is, \begin{equation}\label{eqn:equiMFLD}
    F'(\muinf) + \beta^{-1}\log \muinf = \textrm{const.}\quad\text{on } \R^d.
\end{equation} Moreover, the second variation $F''(\muinf)$ of $F$ at $\muinf$ satisfies
\begin{equation}\label{eqn:2ndvarMFLD}
        \iint F''(\muinf)(x,x')\,f(x)f(x')\,\dd\muinf(x)\dd\muinf(x') \ge -\tau_0\int f^2\,\dd\muinf
    \end{equation}
    for all $f\in L^2(\muinf)$ with $\int f\,\dd\muinf=0$, for some $0\le\tau_0<\beta^{-1}$. Furthermore, $\muinf$ satisfies a weighted Poincar\'e inequality of the form \eqref{eqn:cwpx}, together with the decay conditions \eqref{eqn:bddmomentorig} and \eqref{eqn:tailequil}.
\end{assumption}

\begin{remark}
One concrete example satisfying Assumption~\ref{ass:MFLD} is the interaction energy with weak confinement
\begin{equation}\label{eqn:MVF}
    F(\mu)=\int V(x)\,\dd\mu(x)+\frac12\iint W(x-y)\,\dd\mu(x)\dd\mu(y),
\end{equation}
where $V(x)=a\log\bangle{x}$ for some $a>0$ and $W(z)=\exp\!\big(-\tfrac{|z|^2}{2}\big)$. The Gaussian kernel is positive definite, so the interaction energy is convex in the linear sense, and $F$ is displacement smooth. By \eqref{eqn:equiMFLD}, a stationary measure satisfies
\[
\muinf(x)\propto \bangle{x}^{-a\beta}\exp\!\big(-\beta\,W*\muinf(x)\big),
\]
which defines a probability density provided $a\beta>d$. Since $0\le W*\muinf\le 1$, we have $e^{-\beta}\le \exp(-\beta\,W*\muinf)\le 1$, and therefore $\muinf(x)\sim\bangle{x}^{-a\beta}$ as $|x|\to\infty$. In this case the standard Poincar\'e inequality fails, because $\muinf$ is not exponentially integrable, while a weighted Poincar\'e inequality holds with $\sigma_x=1$.
\end{remark}
\begin{thm}\label{thm:conv-MFLD}
    Let $\mu_t$ be a solution of the MFLD \eqref{eqn:MFLDorig} with initial condition $\mu_0$, and define $f_t:=\frac{\dd\mu_t}{\dd\muinf}-1$. Write $k(x,x'):=F''(\muinf)(x,x')$, and let $\calR:\mathcal{P}_2(\R^d)\to\R$ be the remainder functional defined by
    \begin{equation}\label{eqn:varF}
        F(\mu)=F(\muinf)-\beta^{-1}\int\log\muinf\,\dd(\mu-\muinf)+\frac12\iint k(x,x')\,\dd(\mu-\muinf)(x)\dd(\mu-\muinf)(x')+\calR(\mu)
    \end{equation}
    for all $\mu$. (The first-order term takes this form by \eqref{eqn:equiMFLD}.) Assume in addition that
    \begin{equation}\label{eqn:operator}
        \sup_{x,y}\|\nabla_x\nabla_y F''(\muinf)\|_{op}<\infty,\quad
        \sup_{\substack{x,y,z\in\R^d\\ \tilde\mu\in\mathcal{P}_2(\R^d)}}\|\nabla_x\nabla_y\nabla_z\calR'''(\tilde\mu)\|_{op}<\infty,\quad
        \sup_{\substack{x,y\in\R^d\\ \tilde\mu\in\mathcal{P}_2(\R^d)}}\|\nabla_x\nabla_y^2\calR''(\tilde\mu)\|_{op}<\infty.
    \end{equation}
    Suppose Assumptions~\ref{ass:wpi} (with $\alpha>6\sigma_x$) and~\ref{ass:MFLD} hold, and fix an energy level $R>0$. Then there exists $\varepsilon>0$, depending only on $\alpha,\,\sigma_x,\,R,\,\cwp^x,\,\beta,\,\tau_0$ and the operator bounds in \eqref{eqn:operator}, such that the following holds. If the initial perturbation satisfies
    \begin{equation}\label{eqn:bddmomentMFLD}
        \int f_0^2\bangle{x}^\alpha\,\dd\muinf(x)<\frac R2, \qquad \text{and} \qquad \|f_0\|^2_{L^2(\muinf)}\le\varepsilon,
    \end{equation}
    then the solution $\mu_t$ satisfies the uniform-in-time moment bound
    \begin{equation}\label{eqn:uniform-Ealpha-bound-MFLD}
        \sup_{t\ge0}\int f_t^2\bangle{x}^\alpha\,\dd\muinf(x)\le R.
    \end{equation}
    Moreover, $f_t\to0$ at an algebraic rate: there is a constant $C$, depending on the same parameters, such that
    \begin{equation}\label{eqn:At_upper_decay_MFLD}
        \|f_t\|^2_{L^2(\muinf)}\le C\big(\|f_0\|^2_{\dotH^{-1}(\muinf)}+t\big)^{-\frac{\alpha-2\sigma_x}{2\sigma_x}}.
    \end{equation}
    In particular, if $F$ is quadratic, that is, $\calR=0$, the same estimates hold under the weaker condition $\alpha>4\sigma_x$.
\end{thm}
Similarly to \Cref{thm:lowerbd}, we also have a lower bound for MFLD.
\begin{thm}\label{thm:lowerbd-MFLD}
    Suppose Assumptions~\ref{ass:wpi},~\ref{ass:twosided}, and~\ref{ass:MFLD} hold. Then for every sufficiently large $T>0$ there exists an initial condition $f_0^{(T)}$ with $\int f_0^{(T)}\,\dd\muinf=0, ~ 1+f_0^{(T)}\ge0$,
    which satisfies the hypotheses \eqref{eqn:bddmomentMFLD} of \Cref{thm:conv-MFLD}. Moreover, there is a constant $c>0$, independent of $T$, such that the corresponding solution $f_t^{(T)}$ of \eqref{eqn:MFLDres} satisfies
    \begin{equation}\label{eqn:lower_bds-MFLD}
        \|f_t^{(T)}\|^2_{L^2(\muinf)}\ge c\,(1+t)^{-\frac{p-d}{2}} \qquad \text{for all } t\in[T/2,T].
    \end{equation}
\end{thm}

\subsection{Related Works}\label{subsec:references}
There has been abundant work on Wasserstein GDA flows since their formulation in \cite{hsieh2019finding, domingo2020mean}. In \cite{ma2022provably}, a quasistatic two-timescale limit is considered, in which one dynamics runs infinitely faster or slower than the other, so that long-time convergence reduces to the convergence of two decoupled flows. This assumption was weakened in \cite{lu2023two}, where the timescale separation is large but finite; exponential convergence is established on compact domains, where logarithmic Sobolev inequalities (LSI) hold. The work \cite{an2025convergence} uses a hypocoercive structure to prove convergence for quadratic games in the two-timescale regime, quantifying the scaling of the rates in different timescale regimes, and uses a coupling argument for the intermediate regime. Other modifications of the dynamics include \cite{kim2024symmetric, lu2025convergence}, which use time averaged gradients. For non-Wasserstein flows of two-player games, see \cite{lascu2024fisher} for Fisher--Rao flows and \cite{lascu2025entropic} for best-response flows. In the single-timescale setting considered here, convergence is known under a global convexity-concavity assumption on $k$ \cite{conger2024coupled, cai2024convergence}. See also \cite{OJMO_2025__6__A1_0} for a particle algorithm, \cite{conger2025monotone} for multispecies flows, and \cite{liu2025convergence} for discretization analysis.

A second line of literature close to this work concerns subexponential convergence rates obtained from weak functional inequalities. Besides the weighted Poincar\'e inequality, the \emph{weak Poincar\'e inequality} \cite{rockner2001weak} has also been considered; it allows one to study the $L^2$ convergence rate of reversible dynamics (in particular, the Fokker--Planck equation) with $L^\infty$ initial data. The relations between different weak functional inequalities are studied extensively in \cite{cattiaux2010functional}; in particular, the weighted Poincar\'e inequality implies the weak Poincar\'e inequality. Analogous results for the Fokker--Planck equation with non-gradient drift are given in \cite{kavian2021fokker}. For kinetic Fokker--Planck equations with weak confinement, see \cite{cao2020kinetic, hu2019subexponential, grothaus2019weak, bouin2023mathrm, brigati2024explicit}. For convergence lower bounds for ergodic Markov processes, we refer to \cite{brevsar2026subexponential}.

MFLD has also been well studied since \cite{hu2021mean}, motivated in particular by the mean-field description of noisy neural network training \cite{mei2018mean, chizat2018global}. Simple cases in which convergence is well understood include overdamped Langevin dynamics \cite{bakry_analysis_2014}, where $F(\mu)=\int V\,\dd\mu$ is linear, and the case where $F$ is displacement convex \cite{ambrosio2008gradient}. In \cite{chizat2022meanfield, nitanda2022convex}, a uniform LSI is identified, leading to global exponential convergence in entropy; this was improved to $L^p$ convergence in \cite{chen2025uniform}. MFLD is closely connected to McKean-Vlasov equations, especially when $F$ has the form of \eqref{eqn:MVF}. One work that was established primarily for McKean-Vlasov equations is \cite[Theorem 1.1(b)]{carrillo2020long}, where the equation is posed on the torus and has no external potential. For local convergence, prior to \cite{wang2026local}, \cite{cormier2025stability} gave a criterion for exponential local stability of MFLD with quadratic $F$ in the 1-Wasserstein distance, and \cite{monmarche2025local} established local exponential convergence under a ``local nonlinear LSI'' when $F$ is not convex and multiple stationary states exist. As the name suggests, MFLD is also the mean-field limit of a particle system, and long-time convergence enables uniform-in-time propagation of chaos; see \cite{chen2025uniform, suzuki2023uniform} for continuous-time dynamics, \cite{suzuki2023convergence} for discretizations, and \cite{chewi2026uniform} for a uniform-in-$N$ LSI.
\subsection*{Notation}
Throughout this work, we mostly follow the notation of \cite{wang2026local}. We write $f_t:=\frac{\dd\mu_t}{\dd\muinf}-1, ~~ g_t:=\frac{\dd\nu_t}{\dd\nuinf}-1$ for the relative density perturbations, and study the convergence of $f_t$ and $g_t$ to zero.

For two nonnegative quantities $A$ and $B$, we write $A\lesssim B$ if $A\le CB$ for a constant $C>0$ that does not depend on $t$ or on the initial data. We write $A\asymp B$ if both $A\lesssim B$ and $B\lesssim A$ hold.

We define the bilinear form $\mathcal{K}:L^2(\muinf)\times L^2(\nuinf)\to\R$ by
\[
\mathcal{K}(f,g)=\int_{\cY}\Big(\int_{\cX}k(x,y)f(x)\,\dd\muinf(x)\Big)g(y)\,\dd\nuinf(y),
\]
and the linear operators $K:L^2(\nuinf)\to L^2(\muinf)$ and $K^\top:L^2(\muinf)\to L^2(\nuinf)$ by
\begin{equation}\label{eqn:defK}
    (Kg)(x)=\int_{\cY}k(x,y)g(y)\,\dd\nuinf(y), \qquad (K^\top f)(y)=\int_{\cX}k(x,y)f(x)\,\dd\muinf(x).
\end{equation}
Let $\nabla_x^*:=-\frac{1}{\muinf}\nabla_x\cdot(\muinf\,\bullet)$ denote the adjoint of $\nabla_x$ in $L^2(\muinf)$, and let $L_x:=\nabla_x^*\nabla_x$ be the corresponding elliptic operator. Similarly, let $\nabla_y^*:=-\frac{1}{\nuinf}\nabla_y\cdot(\nuinf\,\bullet)$ and $L_y:=\nabla_y^*\nabla_y$. For moment estimates, we use the Japanese bracket $\bangle{x}:=(1+|x|^2)^{1/2}$. We write $\mathcal{P}_2(\R^d)$ the set of probability measures on $\R^d$ with finite second moments.

\subsection*{Acknowledgment}
The work of J. An was partially supported by NSFC Grant No.12601384. Part of this work is completed during the authors' mutual visits to SJTU and NUS, and we would like to thank SJTU and NUS for their hospitality.

\subsection*{Use of Artificial Intelligence}
OpenAI ChatGPT is used in this work for literature search, formulating some of the proofs and drafting the manuscript. The important ideas, in particular in \Cref{sec:wPIconv}, including combining weighted Poincar\'e inequality with the estimates in \cite{wang2026local}, and the bootstraping argument, are found by the authors. The authors reviewed and edited the content generated by artificial intelligence as needed and take full responsibility for the content of the manuscript.

\section{Proof of \Cref{thm:decay_upper_bd}: Local algebraic convergence}\label{sec:wPIconv}
Following \cite{wang2026local}, we rewrite \eqref{eqn:GDA}, using \eqref{eqn:equilibrium}
    \begin{align*}
  \d_t \mu_t &= \nabla_x \cdot\Big(\mu_t \int_{\cY} \nabla_x k(x,y) \dd\nu_t(y)\Big)+ \beta^{-1}\Delta_x \mu_t\\
  & \quad -\nabla_x \cdot\Big(\mu_t \int_{\cY} \nabla_x k(x,y) \dd\nu_{\infty}(y)\Big)-\beta^{-1} \nabla_x\cdot (\mu_t \nabla_x \log \mu_{\infty})\\
  &=\nabla_x \cdot\Big(\mu_t \int_{\cY} \nabla_x k(x,y) \dd(\nu_t-\nu_{\infty})(y)\Big)+\beta^{-1} \nabla_x\cdot (\mu_t \nabla_x \log \frac{\dd\mu_t}{\dd\mu_{\infty}}).
    \end{align*}
Divide above by $\muinf$ and we obtain a PDE for $f_t$ and repeat the procedure for $\nu_t$, we derive the equations, see also \cite[(4.2)]{wang2026local} %(temporarily setting $\eta=1$)
\begin{equation}\label{eqn:modGDA}
\begin{aligned}
    \frac{\d}{\d t} f_t &= -\beta^{-1} L_x f_t - L_x Kg_t - \nabla_x^*(f_t\nabla_x K g_t),\\
    \frac{\d}{\d t} g_t &= -\beta^{-1} L_y g_t + L_y K^\top f_t + \nabla_y^*(g_t\nabla_y K^\top f_t).
    \end{aligned}
\end{equation}

The important part of this proof is to consider the interplay between these four quantities:
\begin{equation*}
 \begin{aligned}
        Z_t&:=\|f_t\|^2_{\dotH^{-1}(\muinf)}+\|g_t\|^2_{\dotH^{-1}(\nuinf)}, \\A_t&:=\|f_t\|^2_{L^2(\muinf)}+\|g_t\|^2_{L^2(\nuinf)}, \\ B_t&:=\|\nabla_x f_t\|^2_{L^2(\muinf)}+\|\nabla_y g_t\|^2_{L^2(\nuinf)}, \\
         E_{\alpha}(t)&:=\int f_t^2 \bangle{x}^\alpha \dd\mu_\infty(x) +\int g_t^2 \bangle{y}^\alpha \dd\nu_\infty(y),\\
    \end{aligned}
\end{equation*}
The first three quantities are already considered in \cite{wang2026local}; in particular, $Z_t$ is defined using homogeneous Sobolev norm $\dotH^{-1}$, which is different from \cite{wang2026local} since it is no longer directly controlled by the $L^2$ norm.  The last quantity is necessary in our weak confinement setting, following the discussions after \Cref{thm:decay_upper_bd}. Our goal is to show, under Assumption \ref{ass:wpi}, the convergence of $A_t$ to zero as well as its convergence rate.

Following identical arguments to the proof of \cite[Theorem 4.1]{wang2026local}, we find exactly the same differential inequalities 
\begin{equation}\label{eqn:wangchizatkey}
    \frac{\dd Z_t}{\dd t} \le -2\beta^{-1} A_t+2MZ_t\sqrt{A_t}, \qquad \frac{\dd A_t}{\dd t}\le -2\beta^{-1}B_t+2M\sqrt{B_t Z_t}(1+\sqrt{A_t}).
\end{equation}
We note that these two inequalities only use \cite[Lemma 3.9]{wang2026local} and do not require Poincar\'e inequality, which closes the argument after substituting $c^{-1} Z_t \le A_t \le cB_t$.

%A useful observation which we will use repeatedly is that by H\"older's inequality, if $E_{\alpha}(t)$ remains bounded, then 
%    \[ \|f_t\|^2_{L^2(\muinf)} \le \Big(\int f_t^2\bangle{x}^\alpha \dd\muinf \Big)^\frac{2\sigma_x}{\alpha+2\sigma_x}\Big(\int \frac{f_t^2}{\bangle{x}^{2\sigma_x}} \dd\muinf \Big)^\frac{\alpha}{\alpha+2\sigma_x} \lesssim \|\nabla_x f_t\|_{L^2(\muinf)}^\frac{2\alpha}{\alpha+2\sigma_x}.\]

\begin{comment}
The remaining proof relies on a bootstrap argument. Suppose, given a large constant $R>0$ and a small constant $a_0>0$, there is a $T>0$ such that for all $t\in [0,T]$, we have 
\begin{equation}
    E_{\alpha}(t)\leq R,\quad A_t\leq a_R.
\end{equation}
Under $E_{\alpha}(t)\leq R$
\end{comment}

\begin{proof}
Define the weighted dissipation
\[
\begin{aligned}
    D_\alpha(t)
    &:=
    \int \bangle{x}^{\alpha}|\nabla_x f_t|^2\,\dd\mu_\infty
    +
    \int f_t^2\bangle{x}^{\alpha-2}\,\dd\mu_\infty  \\
    &\quad
    +
    \int \bangle{y}^{\alpha}|\nabla_y g_t|^2\,\dd\nu_\infty
    +
    \int g_t^2\bangle{y}^{\alpha-2}\,\dd\nu_\infty .
\end{aligned}
\]
We first claim that, for all sufficiently small local solutions, there exists small constant $c_0>0$ and large constant $C_0>0$ such that 
\begin{equation}
\label{eqn:Ealpha-diff-ineq}
    \frac{\dd}{\dd t}E_\alpha(t)
    +
    c_0D_\alpha(t)
    \le
    C_0 A_t
    +
    C_0 Z_t\bigl(1+E_\alpha(t)\bigr).
\end{equation}
Indeed, differentiating the weighted moment of $f_t$ gives
    \begin{align*}
        \frac{\dd}{\dd t} \frac{1}{2}\int f_t^2\bangle{x}^\alpha \dd\muinf & = \int f_t \bangle{x}^\alpha\Big( -\beta^{-1} L_x f_t - L_x Kg_t - \nabla_x^*(f_t\nabla_x K g_t)\Big) \dd\muinf \\ & = -\beta^{-1}\int \bangle{x}^\alpha |\nabla_x f_t|^2 \dd\muinf -\frac{\alpha}{\beta}\int \bangle{x}^{\alpha-2}f_t x\cdot \nabla_x f_t \dd\muinf \\ & \quad -\int (1+f_t)\nabla_x(f_t\bangle{x}^\alpha) \cdot \nabla_x(Kg_t) \dd\muinf \\ & = -\beta^{-1}\int \bangle{x}^\alpha |\nabla_x f_t|^2 \dd\muinf +\frac{\alpha}{2\beta}\int f_t^2 \nabla_x \cdot(x\bangle{x}^{\alpha-2}) \dd\muinf \\ & \quad + \frac{\alpha}{2\beta} \int f_t^2 \bangle{x}^{\alpha-2} x\cdot \nabla \log \muinf \dd\muinf  -\alpha\int f_t(1+f_t)\bangle{x}^{\alpha-2}x \cdot \nabla_x(Kg_t) \dd\muinf \\ & \quad -\int \bangle{x}^\alpha (1+f_t)\nabla_x f_t\cdot \nabla_x(Kg_t) \dd\muinf  := I_{good} +I_{pert}
    \end{align*}
Here we call the collection of the first three terms ``good terms'' and the latter two terms ``perturbation terms''. Had $f_t$ solved the Fokker-Planck equation (that is, $g_t=0$ in the first equation of \eqref{eqn:modGDA}), only the ``good terms'' remain in the moment estimate. The non-perturbative good part $I_{good}$ satisfies, by \eqref{eqn:tailequil}, $x\cdot \nabla \log \muinf(x) \le a\1_{|x|\le r_0}-p$, when $p\geq\max\{d_x, d_y\}+\alpha$, there exists $0
<c\leq \min\{1/\beta, \alpha/\beta, \alpha^2/2\beta\}$ such that 
\begin{equation}\label{eqn:estIgood}
\begin{aligned}
   I_{good} = &-\beta^{-1}
    \int \bangle{x}^{\alpha}|\nabla_x f_t|^2\,\dd\mu_\infty
    +
    \frac{\alpha}{2\beta}
    \int f_t^2\bangle{x}^{\alpha-2}
    \left[
        d_x+(\alpha-2)\frac{|x|^2}{\bangle{x}^2}
        +x\cdot\nabla\log\mu_\infty
    \right]\dd\mu_\infty                                    \\
    &\qquad\le
    -c
    \int \bangle{x}^{\alpha}|\nabla_x f_t|^2\,\dd\mu_\infty
    -c
    \int f_t^2\bangle{x}^{\alpha-2}\,\dd\mu_\infty
    +
    C_{r_0}\|f_t\|_{L^2(\mu_\infty)}^2 .
\end{aligned}
\end{equation}
The last contribution arises from the $a\1_{|x|\le r_0}$ part of \eqref{eqn:tailequil} and uses $\bangle{x} \le \max\{1,r_0\}$ when $|x|\le r_0$. This positive contribution can be controlled by $A_t$. The same estimate holds for the $g_t$-equation.

It remains to estimate the perturbative terms in $I_{pert}$. An important estimate from \cite[Lemma 3.9]{wang2026local} is that 
    \begin{equation}\label{eqn:nablaKft}
        \sup_{\mathbb{R}^{d_y}} \|\nabla (Kg_t)\| \le M \|g_t\|_{H^{-1}(\nuinf)},
    \end{equation}  with $M = \sup_{x,y}\|\nabla_x \nabla_y k\|_{op}$.
Similarly, we also have $\sup_{\mathbb{R}^{d_x}} \|\nabla (K^\top f_t)\| \le M \|f_t\|_{H^{-1}(\muinf)}$. Therefore, using the zero-average conditions on \(f_t,g_t\), we have
\[
    \|\nabla_x (K g_t)\|_{L^\infty_x}^2
    +
    \|\nabla_y (K^\top f_t)\|_{L^\infty_y}^2
    \le CZ_t .
\]
By Cauchy's inequality and Young's inequality, as well as $\int \bangle{x}^{\alpha} \dd\muinf <+\infty$, 
\[
\begin{aligned}
    &I_{pert}\leq C\sqrt{Z_t}\Big(\int f_t\bangle{x}^{\alpha-1}\dd\muinf+\int f_t^2\bangle{x}^{\alpha-1}\dd\muinf+\int|\nabla_x f_t|(1+f_t)\bangle{x}^{\alpha}\dd\muinf \Big)                                                 \\
    &\qquad\le
    \varepsilon
    \left(
        \int \bangle{x}^{\alpha}|\nabla_x f_t|^2\,\dd\mu_\infty
        +
        \int f_t^2\bangle{x}^{\alpha-2}\,\dd\mu_\infty
    \right)
    +
    C_\varepsilon Z_t
    \left(
        1+
        \int f_t^2\bangle{x}^{\alpha}\,\dd\mu_\infty
    \right).
\end{aligned}
\]
The analogous estimate for the \(g_t\)-equation is
\[
\begin{aligned}
    &\left|
    \int (1+g_t)\nabla_y(g_t\bangle{y}^{\alpha})
    \cdot\nabla_y (K^\top f_t)\,\dd\nu_\infty
    \right|                                                   \\
    &\qquad\le
    \varepsilon
    \left(
        \int \bangle{y}^{\alpha}|\nabla_y g_t|^2\,\dd\nu_\infty
        +
        \int g_t^2\bangle{y}^{\alpha-2}\,\dd\nu_\infty
    \right)
    +
    C_\varepsilon Z_t
    \left(
        1+
        \int g_t^2\bangle{y}^{\alpha}\,\dd\nu_\infty
    \right).
\end{aligned}
\]
We then choose $\varepsilon\in (0,c/2)$ sufficiently small, and combined with \eqref{eqn:estIgood}, we obtain \eqref{eqn:Ealpha-diff-ineq}.

We now close the estimate by a bootstrap argument. Set
\begin{equation} \label{eqn:Reps}
    R>2E_\alpha(0)+\delta,
\end{equation}
without loss of generality, we may assume $\delta\in (0,1)$. For simplicity we also assume $\sigma_x\geq \sigma_y>0$. Let \(T>0\) be such that, on \([0,T]\),
\begin{equation}
\label{eqn:bootstrap-assumption}
    E_\alpha(t)\le R.
\end{equation}
To control $Z_t$, we use the dual formula of the $\dotH^{-1}$ norm
\begin{equation}\label{def:H-1}
    \|f_t\|_{\dot H^{-1}(\mu_\infty)}
=\sup_{\substack{
\varphi\in \dot H_0^1(\mu_\infty)\\
\|\varphi\|_{\dot H_0^1(\mu_\infty)}=1
}}
\left|
\int f_t \varphi \, \dd\mu_\infty
\right|.
\end{equation}
We take any test function $\varphi \in \dotH^1(\muinf)$ with $\int \varphi \dd \muinf=0$ and $\|\varphi\|_{\dotH^1_0(\muinf)}=1$, and repeatedly use H\"older's inequality:
    \begin{align*}
    \int f_t\varphi \dd\muinf & \le \Big(\int f_t^2\bangle{x}^{2\sigma_x} \dd\muinf\Big)^\frac{1}{2}\Big(\int \frac{\varphi^2}{\bangle{x}^{2\sigma_x}}\dd\muinf\Big)^\frac{1}{2} \\ & \leq C \Big(\int f_t^2\bangle{x}^{2\sigma_x} \dd\muinf\Big)^\frac{1}{2}  \\ & \le C\Big(\int f_t^2\bangle{x}^\alpha \dd\muinf \Big)^{\frac{\sigma_x}{\alpha}} \Big(\int f_t^2\dd\muinf \Big)^{\frac{1}{2}-\frac{\sigma_x} {\alpha}}\lesssim R^{\frac{\sigma_x}{\alpha}} \Big(\int f_t^2\dd\muinf \Big)^{\frac{1}{2}-\frac{\sigma_x}{\alpha}}.
    \end{align*}
    Here the second inequality uses
    \[ \int \frac{\varphi^2}{\bangle{x}^{2\sigma_x}} \dd \muinf \stackrel{\eqref{eqn:cwpx}}{\lesssim} \int |\nabla_x \varphi|^2 \dd \muinf =1.\]
    This yields by duality \eqref{def:H-1} \[\|f_t\|_{\dotH^{-1}(\muinf)} \lesssim R^{\frac{\sigma_x}{\alpha}} \Big(\int f_t^2\dd\muinf \Big)^{\frac{1}{2}-\frac{\sigma_x}{\alpha}} .\] Similarly, $ \|g_t\|_{\dotH^{-1}(\nuinf)}\lesssim R^{\frac{\sigma_x}{\alpha}} \Big(\int g_t^2\dd\nuinf \Big)^{\frac{1}{2}-\frac{\sigma_x}{\alpha}}$.
    As a consequence, we get for all $t\in [0,T]$
    \begin{equation}\label{eqn:Zt_At}
        Z_t\lesssim R^{\frac{2\sigma_x}{\alpha}} A_t^{1-\frac{2\sigma_x}{\alpha}}.
    \end{equation}
    We are now ready to prove local algebraic convergence rates. From the above inequality, we have
    \begin{equation}\label{eqn:Z_split}
        Z_t = Z_t^{1-\frac{\alpha}{2(\alpha-2\sigma_x)}}Z_t^{\frac{\alpha}{2(\alpha-2\sigma_x)}}\lesssim  R^{\frac{\sigma_x}{\alpha-2\sigma_x}} Z_t^{1-\frac{\alpha}{2(\alpha-2\sigma_x)}}\sqrt{A_t}
    \end{equation} 
Now  if $\alpha>4\sigma_x$, then substitute \eqref{eqn:Z_split} into the first inequality of \eqref{eqn:wangchizatkey}, and we find that there exists a small constant  $c>0$ such that 
    \begin{equation}\label{eqn:dZtneg}
        \frac{\dd Z_t}{\dd t} \le -2\beta^{-1} A_t+2CR^{\frac{\sigma_x}{\alpha-2\sigma_x}} Z_t^{1-\frac{\alpha}{2(\alpha-2\sigma_x)}}A_t.
    \end{equation}
    As long as $A_0$ is sufficiently small (the threshold depends on $R$), then so is $Z_0$ thanks to \eqref{eqn:Zt_At}, which makes the right-hand side of \eqref{eqn:dZtneg} negative, and in particular smaller than $-cA_t$ for some $c>0$. Therefore on the bootstrap interval $[0,T]$, $Z_t$ is always decreasing, and \eqref{eqn:Zt_At} gives
    \begin{equation}\label{eqn:dZtlemcA_t}
        \frac{\dd Z_t}{\dd t}\le -cA_t \lesssim - R^{-\frac{\alpha-2\sigma_x}{2\sigma_x}} Z_t^{\frac{\alpha}{\alpha-2\sigma_x}},
    \end{equation} which leads to (note that $\frac{\alpha-2\sigma_x}{2\sigma_x}>1$)
    \begin{equation}\label{eqn:Zt_decay}
        Z_t \le \Big(\frac{2\sigma_x}{\alpha-2\sigma_x}cR^{-\frac{\alpha-2\sigma_x}{2\sigma_x}} t+ Z_0^{-\frac{2\sigma_x}{\alpha-2\sigma_x}}\Big)^{-\frac{\alpha-2\sigma_x}{2\sigma_x}}\sim \big(Z_0^{-\frac{2\sigma_x}{\alpha-2\sigma_x}}+C_Rt\big)^{-\frac{\alpha-2\sigma_x}{2\sigma_x}}, \quad \text{for}~t\in[0,T].
    \end{equation}
    Here and below all constants depend on $R$. We now return to \eqref{eqn:Ealpha-diff-ineq}. Dropping the nonnegative dissipation term and applying Gr\"onwall's inequality gives that for $C_0$ that appears in \eqref{eqn:Ealpha-diff-ineq},
\[
\begin{aligned}
    E_\alpha(t)
    &\le
    \left(
        E_\alpha(0)
        +
        C_0\int_0^t A_s\,ds
        +
        C_0\int_0^t Z_s\,ds
    \right)
    \exp\left(
        C_0\int_0^t Z_s\,ds
    \right).
\end{aligned}
\]
Thanks to $\frac{\alpha-2\sigma_x}{2\sigma_x}>1$, the right-hand side of \eqref{eqn:Zt_decay} is integrable as $t\in (0,\infty)$. Meanwhile, the relationship $\frac{\dd Z_t}{\dd t} \le -cA_t$ yields $\int_0^t A_s \dd s \lesssim Z_0$. We therefore obtain
\[
    E_\alpha(t)
    \le
    \left(
        E_\alpha(0)
        +C_0 Z_0
        +\frac{2\sigma_x}{\alpha-4\sigma_x}C_0Z_0^{\frac{\alpha-4\sigma_x}{\alpha-2\sigma_x}}
    \right)
    \exp\left(\frac{2\sigma_x}{\alpha-4\sigma_x}C_0 Z_0\right),
    \qquad 0\le t\le T .
\]
 Since \(R>2E_\alpha(0)\), we can choose a  constant $\varepsilon>0$ depending on $\alpha, \sigma_*,R,C_0$ small enough so that \(A_0\leq \varepsilon\) implies (note that $Z_0$ is also sufficiently small due to \eqref{eqn:Zt_At})
\[
   \left(
        E_\alpha(0)
        +C_0 Z_0
        +\frac{2\sigma_x}{\alpha-4\sigma_x}C_0Z_0^{\frac{\alpha-4\sigma_x}{\alpha-2\sigma_x}}
    \right)
    \exp\left(\frac{2\sigma_x}{\alpha-4\sigma_x}C_0 Z_0\right) \leq
    \frac R2 .
\]
Thus the bootstrap bound \(E_\alpha(t)\le R\) improves to \(E_\alpha(t)\le R/2\).
Hence the bootstrap closes globally, and
\[
    \sup_{t\ge 0}E_\alpha(t)\le R.
\]   
    
    We finally control $A_t$. By Young's inequality, we obtain from the second inequality of \eqref{eqn:wangchizatkey}
    \[ \frac{\dd A_t}{\dd t} \le -\beta^{-1}B_t + CZ_t(1+A_t) \le CZ_t(1+E_\alpha(t)) \le CRZ_t.\]
Here we are using $A_t \le E_\alpha(t)\le R$. For $t\ge 1$ and $s\in [t-1,t]$, we obtain by integrating the above inequality in time, using monotonicity of $Z_t$
\[A_t \le A_s + CR\int_s^t Z_u \dd u \le A_s+CRZ_{t-1}. \]
Averaging over $s\in [t-1,t]$, 
and combining with \eqref{eqn:dZtlemcA_t}, $\frac{\dd Z_t}{\dd t} \le -cA_t$, we find
\[ A_t \le \int_{t-1}^t A_s \dd s +CRZ_{t-1} \le c^{-1}(Z_{t-1}-Z_t) +CRZ_{t-1} \le CRZ_{t-1}.\]
Together with \eqref{eqn:Zt_decay}, this finishes the proof.

\end{proof}

\section{Proof of \Cref{thm:lowerbd}: A lower bound}
\label{sec:lower-bound}

In this section, we construct slowly relaxing perturbations for the full
coupled dynamics~\eqref{eqn:modGDA}. The purpose is to give a lower bound on
the worst-case convergence rate allowed by the assumptions. More precisely,
for every sufficiently large observation time $T$, we construct an admissible
initial condition, with the same uniform moment and smallness bounds as in
\Cref{thm:decay_upper_bd}, whose solution has not decayed too much by time $T$.

Notice that, with $h=(f,g)^\top$, equation~\eqref{eqn:modGDA} can be written as
\begin{equation}
    \partial_t h=-\mathcal A h+\mathcal N(h),
    \label{eqn:lower-bound-abstract-evolution}
\end{equation}
where
\[
    \mathcal A
    :=
    \begin{pmatrix}
        \beta^{-1}L_x & L_xK\\
        -L_yK^\top & \beta^{-1}L_y
    \end{pmatrix},
    \qquad
    \mathcal N(f,g)
    :=
    \begin{pmatrix}
        -\nabla_x^*(f\nabla_xKg)\\
        \nabla_y^*(g\nabla_yK^\top f)
    \end{pmatrix}.
    \]
We work in $\mathcal H
    :=
    \dot H^{-1}(\mu_\infty)
    \oplus
    \dot H^{-1}(\nu_\infty),$
with inner product
\[
    \langle h,\widetilde h\rangle_{\mathcal H}
    =
    \langle f,L_x^{-1}\widetilde f\rangle_{L^2(\mu_\infty)}
    +
    \langle g,L_y^{-1}\widetilde g\rangle_{L^2(\nu_\infty)}.
    \]
The off-diagonal part of $\mathcal A$ is skew-adjoint in $\mathcal H$, and
hence
\begin{align*}
    \langle \mathcal Ah,h\rangle_{\mathcal H}
    &=
    \beta^{-1}
    \left(
        \|f\|_{L^2(\mu_\infty)}^2
        +
        \|g\|_{L^2(\nu_\infty)}^2
    \right).
\end{align*}
Moreover, the estimate in \eqref{eqn:nablaKft} gives
\begin{equation}
    \|L_xKg\|_{\dot H^{-1}(\mu_\infty)}
    +
    \|L_yK^\top f\|_{\dot H^{-1}(\nu_\infty)}
    \leq
    C M\|h\|_{\mathcal H}.
    \label{eqn:coupling-bounded-Hminus1}
\end{equation}
Consequently, $-\mathcal A$ generates a contraction semigroup $S(t):=e^{-t\mathcal A}$ on $\mathcal H$, that is, $\|S(t)\|_{\mathcal{H}\to \mathcal{H}}\leq 1$.

The idea of the proof is to construct an initial perturbation supported in the far field where the equilibrium has little mass. Its width is proportional to its distance \(r\) from the origin, giving a diffusive relaxation time of order \(r^2\), while the small equilibrium mass makes the perturbation admissible and suppresses the nonlinear interaction. We first construct these perturbations and estimate their initial velocity under the linearized dynamics. We then show that the nonlinear solution remains close to its initial condition in \(\mathcal H\) until time \(O(r^2)\). This persistence estimate gives the $\dotH^{-1}$ lower bound directly and the \(L^2\) lower bound by testing against the initial perturbation. Choosing \(r\sim \sqrt T\) produces the claimed bounds on \([T/2,T]\).
\begin{proof}
We divide the proof into four steps. All asymptotics throughout are with respect to $r\to\infty$.

\medskip
\noindent\emph{Step 1: construction of initial condition.}
Fix a unit vector $e_1\in\mathbb R^{d_x}$ and set $x_r:=re_1$, then for all sufficiently
large $r$,
\[ B(x_r,\frac{r}{4})
    \subset
    \left\{x:\frac r2\leq |x|\leq\frac{3r}{2}\right\}.
\]
It follows from~\eqref{eqn:two-sided-polynomial-tail} that
\begin{equation}
    \mu_\infty(x)\asymp r^{-p},
    \qquad x\in B(x_r,\frac{r}{4}),
    \label{eqn:mu-local-comparability}
\end{equation}
where all implicit constants are independent of $r$.  Define
$m_r:=\mu_\infty(B(x_r,\frac{r}{4})),$
then $m_r\asymp r^{-p+d_x}.$ 
Choose a nonnegative function $\psi\in C_c^\infty(B(0,1/8)), \,
    \psi\not\equiv0, $
and fix a unit vector $e\in\mathbb R^{d_x}$. Set
\[
    z_+:=\frac12e,
    \qquad
    z_-:=-\frac12e,
\]
and define
\[
    \psi_{\pm,r}(x)
    :=
    \psi\left(\frac{x-x_r}{r/4}-z_\pm\right).
\]
Their supports are disjoint and both are contained in $B(x_r,\frac{r}{4})$. Let
\[
    M_{\pm,r}
    :=
    \int \psi_{\pm,r}\,\dd\mu_\infty,
    \qquad
    \gamma_r:=\frac{M_{+,r}}{M_{-,r}},
\]
and define
\[
    q_r:=\psi_{+,r}-\gamma_r\psi_{-,r}.
\]
By construction,
\begin{equation}
    \int q_r\,\dd\mu_\infty=0.
    \label{eqn:qr-zero-mean}
\end{equation}
By~\eqref{eqn:mu-local-comparability},
$M_{+,r}\asymp M_{-,r}\asymp m_r, \,
    \gamma_r\asymp1.$ Consequently,
\begin{equation}
    \|q_r\|_{L^\infty}\leq C,
    \qquad
    \|q_r\|_{L^1(\mu_\infty)}\asymp m_r,
    \qquad
    \|q_r\|_{L^2(\mu_\infty)}^2\asymp m_r.
    \label{eqn:qr-Lp-estimates}
\end{equation}
Since
$ \nabla\psi_{\pm,r}(x)
   \asymp r^{-1}
    \nabla\psi\left(\frac{x-x_r}{\ell_r}-z_\pm\right),
$
we also have
$\|\nabla q_r\|_{L^2(\mu_\infty)}^2 \asymp r^{-2}m_r.$
Choose a fixed $\kappa>0$ sufficiently small that $ 1+\kappa q_r\geq0 $
for every sufficiently large $r$; this is possible by the uniform $L^\infty$ bound in~\eqref{eqn:qr-Lp-estimates}. Our choice of initial data $h_{0,r}$ of \eqref{eqn:lower-bound-abstract-evolution} has the form
\[
    h_{0,r}
    :=
    \binom{f_{0,r}}{g_{0,r}}
     =
    \binom{\kappa q_r}{0}:=
    \kappa\Psi_r.
    \]
     Then $A_0\asymp\kappa^2m_r\ll 1$, and because $\bangle{x}\asymp r$ on the support of $q_r$,
\[
    E_\alpha(0)
    \asymp
    \kappa^2r^\alpha m_r
    \asymp
    \kappa^2r^{\alpha-p+d_x}.
    \]
Since $p>d_x+\alpha$, both $A_0$ and
$E_\alpha(0)$ tend to zero as $r\to\infty$. Thus, after fixing $\kappa$
and then taking $r$ sufficiently large, the hypotheses of \Cref{thm:decay_upper_bd}
hold uniformly.

We also need estimates on $\|q_r\|_{\dotH^{-1}(\muinf)}$. Taking
$\phi =\frac{q_r}{\|\nabla q_r\|_{L^2(\mu_\infty)}}$
in the dual characterization~\eqref{def:H-1}, we get
\begin{equation}
\begin{aligned}
    \|q_r\|_{\dot H^{-1}(\mu_\infty)}
    &\geq
    \frac{\|q_r\|_{L^2(\mu_\infty)}^2}
    {\|\nabla q_r\|_{L^2(\mu_\infty)}}
    \gtrsim
    rm_r^{1/2}.
\end{aligned}
    \label{eqn:qr-Hminus1-lower}
\end{equation}
For the reverse inequality, take an arbitrary test function $\phi$ supported in $B(x_r,\frac{r}{4})$, and let
$\overline\phi_r:=\frac{1}{m_r}\int_{B(x_r,\frac{r}{4})}\phi\,\dd\mu_\infty$. Since $q_r$ is supported in $B(x_r,\frac{r}{4})$ and has zero
$\mu_\infty$-mean,
\begin{align*}
    \left|\int q_r\phi\,\dd\mu_\infty\right|
    &=
    \left|
        \int_{B(x_r,\frac{r}{4})}
        q_r(\phi-\overline\phi_r)\,\dd\mu_\infty
    \right|\\
    &\leq
    \|q_r\|_{L^2(\mu_\infty)}
    \|\phi-\overline\phi_r\|_{L^2(B(x_r,\frac{r}{4}),\mu_\infty)}.
\end{align*}
The local comparability~\eqref{eqn:mu-local-comparability} and the standard
Poincar\'e inequality on a Euclidean ball imply
\[
    \|\phi-\overline\phi_r\|_{L^2(B(x_r,\frac{r}{4}),\mu_\infty)}
    \leq
    Cr
    \|\nabla\phi\|_{L^2(B(x_r,\frac{r}{4}),\mu_\infty)}.
\]
Therefore,
\begin{equation}
    \|q_r\|_{\dot H^{-1}(\mu_\infty)}
    \lesssim
    rm_r^{1/2}.
    \label{eqn:qr-Hminus1-upper}
\end{equation}
Combining~\eqref{eqn:qr-Hminus1-lower} and
\eqref{eqn:qr-Hminus1-upper}, we obtain
\begin{equation}
    \|q_r\|_{\dot H^{-1}(\mu_\infty)}^2
    \asymp
    r^2m_r, \qquad \textrm{ and } \qquad Z_0 \asymp \kappa^2 r^2 m_r.
    \label{eqn:qr-Hminus1-equivalence}
\end{equation}

\medskip
\noindent\emph{Step 2: estimates on the linear equation.} Notice that
\[
    \mathcal A\Psi_r
    =
    \binom{\beta^{-1}L_xq_r}{-L_yK^\top q_r}.
    \]
Since $L_x=\nabla_x^*\nabla_x$,
\[
    \|L_xq_r\|_{\dot H^{-1}(\mu_\infty)}
    =
    \|\nabla q_r\|_{L^2(\mu_\infty)}
    \lesssim
    r^{-1}m_r^{1/2}.
    \]
For the other component, the cancellation~\eqref{eqn:qr-zero-mean} gives for every $y$,
\begin{align*}
    \nabla_yK^\top q_r(y)
    &=
    \int \nabla_yk(x,y)q_r(x)\,\dd\mu_\infty(x)\\
    &=
    \int
    \bigl(\nabla_yk(x,y)-\nabla_yk(x_r,y)\bigr)
    q_r(x)\,\dd\mu_\infty(x).
\end{align*}
Using~\eqref{eqn:bddk2}, the support property of $q_r$, and
\eqref{eqn:qr-Lp-estimates}, $|\nabla_yK^\top q_r(y)| \leq CMrm_r,$
which leads to
\[
    \|L_yK^\top q_r\|_{\dot H^{-1}(\nu_\infty)}
    =
    \|\nabla_yK^\top q_r\|_{L^2(\nu_\infty)}
    \lesssim
    Mrm_r.
    \]
Together with~\eqref{eqn:qr-Hminus1-equivalence}, this yields
\begin{equation}
    \frac{\|\mathcal A\Psi_r\|_{\mathcal H}}
    {\|\Psi_r\|_{\mathcal H}}
    \lesssim
    r^{-2}+m_r^{1/2}.
    \label{eqn:full-quasimode-residual}
\end{equation}
Now $r^2m_r^{1/2}
    \asymp
    r^{2-\frac12(p-d_x)}.$
The assumptions $p>d_x+\alpha$ and $\alpha>4\sigma_*=4$ imply $p-d_x>4$, hence $m_r^{1/2}\ll r^{-2},$ and therefore, for all sufficiently large $r$,
\begin{equation}
    \|\mathcal A\Psi_r\|_{\mathcal H}
    \leq
    Cr^{-2}\|\Psi_r\|_{\mathcal H}.
    \label{eqn:quasimode-final}
\end{equation}
% In particular, if
% \begin{equation}
%     \Phi_r:=\frac{\Psi_r}{\|\Psi_r\|_{\mathcal H}},
% \end{equation}
% then
% \begin{equation}
%     \|\Phi_r\|_{\mathcal H}=1,
%     \qquad
%     \|\mathcal A\Phi_r\|_{\mathcal H}\lesssim\ell_r^{-2},
%     \label{eqn:normalized-quasimode}
% \end{equation}
% while
% \begin{equation}
%     \|q_r\|_{L^2(\mu_\infty)}^2
%     \asymp m_r,
%     \qquad
%     \|\Psi_r\|_{\mathcal H}^2
%     \asymp\ell_r^2m_r.
%     \label{eqn:Psi-r-size-summary}
% \end{equation}

% The proof of Proposition~2.2 gives
% \begin{equation}
%     \sup_{s\geq0}Z_s\leq Z_0,
%     \qquad
%     \int_0^\infty A_s\,ds\leq C Z_0.
%     \label{eqn:Z-monotone-A-integrable}
% \end{equation}

 Let $t_r:=c_1r^2,$ where $c_1>0$ will be chosen sufficiently small, independently of $r$.  Let $h_r(t)$ denote the nonlinear solution with initial datum
$h_{0,r}=\kappa\Psi_r$. By Duhamel's formula,
\begin{equation}
    h_r(t)
    =
    \kappa S(t)\Psi_r
    +
    \int_0^tS(t-s)\mathcal N(h_r(s))\,\dd s.
    \label{eqn:nonlinear-Duhamel-lower-bound}
\end{equation}
As long as $c_1$ is sufficiently small, due to the diffusive behavior of $S(t)$, $S(t)\Psi_r$ should be sufficiently close to $\Psi_r$ up to time $t_r$. By contractivity of $S(t)$ and~\eqref{eqn:quasimode-final}, for $0\leq t\leq t_r$,
\begin{equation}
\begin{aligned}
    \|(S(t)-I)\Psi_r\|_{\mathcal H}
    &\leq
    \int_0^{t}
    \|S(s)\mathcal A\Psi_r\|_{\mathcal H}\,ds\leq
    Ctr^{-2}\|\Psi_r\|_{\mathcal H}
    \leq
    Cc_1\|\Psi_r\|_{\mathcal H}.
\end{aligned}
    \label{eqn:linear-persistence-H}
\end{equation}
\begin{comment}
In particular, choosing $c_1$ sufficiently small,
\[
    \|S(t)\Psi_r\|_{\mathcal H}
    \geq
    \frac34\|\Psi_r\|_{\mathcal H},\qquad \text{for}~t\in[0,t_r].
    \]\end{comment}
We also need a componentwise estimate. Since
\[
    \langle u,q_r\rangle_{L^2(\mu_\infty)}
    =
    \left\langle
        \binom{u}{v},
        \binom{L_xq_r}{0}
    \right\rangle_{\mathcal H},
\]
and
\[
    \left\|\binom{L_xq_r}{0}\right\|_{\mathcal H}
    =
    \|\nabla q_r\|_{L^2(\mu_\infty)},
\]
we infer from~\eqref{eqn:linear-persistence-H} that, for $0\leq t\leq t_r$,
\[\left|
        \left\langle
            [S(t)\Psi_r]_1-q_r,q_r
        \right\rangle_{L^2(\mu_\infty)}
    \right|
    \leq
    Cc_1
    \|\Psi_r\|_{\mathcal H}
    \|\nabla q_r\|_{L^2(\mu_\infty)}
    \leq
    Cc_1m_r.
\]
Using~\eqref{eqn:qr-Lp-estimates}, and decreasing $c_1$ if necessary, we
obtain
\begin{equation}
    \left\langle
        [S(t)\Psi_r]_1,q_r
    \right\rangle_{L^2(\mu_\infty)}
    \geq
\frac{3}{4}\|q_r\|_{L^2(\mu_\infty)}^2,\qquad \text{for}~t\in[0,t_r].
    \label{eqn:linear-persistence-pairing}
\end{equation}

\medskip
\noindent\emph{Step 3: estimates on the nonlinear perturbation.}
The contribution from the $\mathcal{N}(h)$ term is treated as the error term, which can be controlled by
\begin{equation}
\begin{aligned}
    \|\mathcal N(h_r(s))\|_{\mathcal H}
    & = \Big(\|\nabla_x^*(f_{r,s}\nabla_x(Kg_{r,s}))\|^2_{\dotH^{-1}(\muinf)} + \|\nabla_y^*(g_{r,s}\nabla_y(K^\top f_{r,s}))\|^2_{\dotH^{-1}(\nuinf)}\Big)^\frac{1}{2} \\ & \le \Big(\|f_{r,s}\nabla_x(Kg_{r,s})\|^2_{L^2(\muinf)} + \|g_{r,s}\nabla_y(K^\top f_{r,s})\|^2_{L^2(\nuinf)}\Big)^\frac{1}{2} \\ & \le \Big( \|f_{r,s}\|_{L^2(\muinf)}^2 \|\nabla_x (Kg_{r,s})\|_{L^\infty}^2 + \|g_{r,s}\|_{L^2(\nuinf)} ^2\|\nabla_y (Kf_{r,s})\|_{L^\infty}^2\Big)^\frac{1}{2}
    \\ & \stackrel{\eqref{eqn:nablaKft}}{\leq}
    C M\sqrt{A_sZ_s}.
\end{aligned}
    \label{eqn:remainder-bound}
\end{equation}
Note that in this time frame, $Z_t$ is decreasing, and $\frac{\dd Z_t}{\dd t} \le -cA_t$, hence \[ \int_0^t Z_s \dd s \le tZ_0, \qquad \int_0^t A_s \dd s \lesssim Z_0,\]
thus by Cauchy-Schwarz,
\begin{equation}\label{eqn:rtAsZs}
    \int_0^t \sqrt{A_sZ_s} \dd s \le \sqrt{t}Z_0.
\end{equation} 
Using~\eqref{eqn:remainder-bound} and \eqref{eqn:rtAsZs}, we get that 
\begin{align*}
    \left\|
        \int_0^{t}S(t-s)\mathcal N(h_r(s))\,\dd s
    \right\|_{\mathcal H}
    \leq CM\int_0^{t}\sqrt{A_sZ_s}\,\dd s\leq
    C M\sqrt{t} Z_0\asymp
    \kappa^2r^3 m_r.
\end{align*}
% Relative to the size of the slow mode, this error is controlled by
% \begin{equation}
%     \frac{\kappa^2\ell_r^3m_r}
%     {\kappa\|\Psi_r\|_{\mathcal H}}
%     \lesssim
%     \kappa\ell_r^2m_r^{1/2}.
%     \label{eqn:nonlinear-relative-error}
% \end{equation}
% By~\eqref{eqn:mr-scaling},
% \begin{equation}
%     \ell_r^2m_r^{1/2}
%     \asymp
%     r^{2\sigma_*-\frac12(p-d_x\sigma_*)}
%     \longrightarrow0.
% \end{equation}
% Thus, for all sufficiently large $r$, the nonlinear error in
% \eqref{eqn:nonlinear-Duhamel-error} is at most
% $\frac14\kappa\|\Psi_r\|_{\mathcal H}$. Combining this with
% \eqref{eqn:linear-persistence-norm}, we obtain
% \begin{equation}
%     \|h_r(t_r)\|_{\mathcal H}
%     \geq
%     \frac12\kappa\|\Psi_r\|_{\mathcal H}.
% \end{equation}
% Therefore,
% \begin{equation}
%     Z_{t_r}
%     \gtrsim
%     \kappa^2\ell_r^2m_r.
%     \label{eqn:Z-lower-at-tr}
% \end{equation}
Now let
$w_r(t):=\int_0^tS(t-s)\mathcal N(h_r(s))\,ds$. By the same duality argument as above,
\begin{align*}
    \left|
        \langle [w_r(t)]_1,q_r\rangle_{L^2(\mu_\infty)}
    \right|
    &\leq
    \|w_r(t)\|_{\mathcal H}
    \|\nabla q_r\|_{L^2(\mu_\infty)}\\
    &\lesssim
   \kappa^2r^2 m_r^{3/2} \ll \kappa m_r \asymp \kappa \|q_r\|^2_{L^2(\muinf)},\qquad \text{as}~r\to \infty.
\end{align*}
% Since
% \begin{equation}
%     \kappa\ell_r^2m_r^{1/2}\longrightarrow0,
% \end{equation}
% this is, for all sufficiently large $r$, bounded by
% \begin{equation}
%     \frac14\kappa\|q_r\|_{L^2(\mu_\infty)}^2.
% \end{equation}

\medskip
\noindent\emph{Step 4: conclusion.} We are now ready to obtain the slow time decay estimates. Multiplying
\eqref{eqn:nonlinear-Duhamel-lower-bound} with $(q_r,0)^\top$, together with~\eqref{eqn:linear-persistence-pairing}, we obtain that for $0\leq t\leq t_r$,
\[
    \langle f_{t,r},q_r\rangle_{L^2(\mu_\infty)}
    \geq \frac{3}{4}\kappa\|q_r\|_{L^2(\mu_\infty)}^2 - c\kappa^2r^2 m_r^{3/2}\geq \frac{1}{2}\kappa\|q_r\|_{L^2(\mu_\infty)}^2 \asymp \kappa m_r.
\]
Thus, by Cauchy-Schwarz inequality,
\[
    \|f_{t,r}\|_{L^2(\mu_\infty)}
    \geq
    \frac{
        |\langle f_{t,r},q_r\rangle_{L^2(\mu_\infty)}|
    }{
        \|q_r\|_{L^2(\mu_\infty)}
    }
    \gtrsim
    \kappa m_r^{1/2}.
\]
Consequently,
\begin{equation}
    A_{t}
    \geq
    \|f_{t,r}\|_{L^2(\mu_\infty)}^2
    \gtrsim
    \kappa^2m_r, \qquad \text{for}~t\in[0,t_r].
    \label{eqn:A-lower-at-tr}
\end{equation}
Recall that $t_r=c_1r^{2}$, and $m_r\asymp r^{-p+d_x}$.
Thus~\eqref{eqn:A-lower-at-tr} gives
\[
    A_{t}
    \gtrsim
    t_r^{-\frac{p-d_x}{2}}\gtrsim t^{-\frac{p-d_x}{2}}, \qquad \text{for }~ t\in[t_r/2, t_r],
\] 
%which implies that
%\[
%    Z_{t}
 %   \gtrsim
  %  t^{-\frac{p-(d_x+2)}{2}},  \qquad \text{for }~ t\in[t_r/2, t_r].
%\]
Finally, for each sufficiently large $T$, we can choose $r=r(T)$ so that
$t_r=T$. This proves~\eqref{eqn:lower_bds}.
\end{proof}

% \begin{corollary}[Polynomial-tail benchmark]
% \label{cor:polynomial-tail-benchmark}
% Assume that $\sigma_*=1$ and
% \begin{equation}
%     p=d_x+\alpha+\delta
% \end{equation}
% for some $\delta>0$. Then, for every sufficiently large $T$, there exists an
% admissible initial condition satisfying the same uniform bounds as in
% Proposition~2.2 such that
% \begin{equation}
%     Z_T
%     \geq
%     c(1+T)^{-\frac{\alpha+\delta-2}{2}},
%     \label{eqn:Z-lower-benchmark}
% \end{equation}
% and
% \begin{equation}
%     A_T
%     \geq
%     c(1+T)^{-\frac{\alpha+\delta}{2}}.
%     \label{eqn:A-lower-benchmark}
% \end{equation}
% In particular, as $\delta\downarrow0$, the exponent in
% \eqref{eqn:Z-lower-benchmark} approaches the upper-bound exponent
% $(\alpha-2)/2$ obtained in Section~2 for $Z_t$.
% \end{corollary}

\begin{remark}[Quantifiers in the lower bound]
The initial condition constructed in \Cref{thm:lowerbd} depends on the observation time $T$. The theorem therefore gives a lower bound on the worst-case convergence rate: in particular, it rules out any faster decay estimate that holds uniformly over the admissible class of initial data. It does not assert that a single fixed initial condition satisfies the same lower bound for all sufficiently large times. Such a fixed-data statement would require a multiscale initial datum, together with additional control of the nonlinear interactions among its different tail packets.
\end{remark}

\section{Results for Mean-field Langevin Dynamics} \label{sec:MFLD}
For this section, since there is no $g$ or $\nu$, we use $K$ to denote the operator $K^\top$ defined in \eqref{eqn:defK}, and also use $L$ to denote $L_x$, and $\sigma$ for $\sigma_x$.
\begin{proof}[Proof of \Cref{thm:conv-MFLD}]
    We again start with writing out the equation for $f_t$, which is present in \cite[(3.5)]{wang2026local}
    \begin{equation}\label{eqn:MFLDres}
      \partial_t f_t=-\beta^{-1}Lf_t-LKf_t-\nabla^*(f_t\nabla Kf_t)-\nabla^*((f_t+1)\nabla \calR'(\mu_t)) 
    \end{equation}
   The treatment of the first three terms is identical to the proof of \Cref{thm:decay_upper_bd}; the third term was changed from $\nabla_x^*(f_t\nabla Kg_t)$ to $\nabla^*(f_t\nabla Kf_t)$, but the analysis of this term only uses \eqref{eqn:nablaKft} with $f_t$ in place of $g_t$, and the estimate is thus also identical. As for the final term, the inequality $W_2^2 (\mu,\muinf) \le 2c^{-1}_{PI}\chi^2 (\mu\|\muinf)$ in \cite{liu2020poincare} no longer works, but we can use \cite[Theorem 2.1]{peyre2018comparison} to bound the term by $\dot{H}^{-1}$ norm:
   \begin{equation}\label{eqn:LinfRpmu}
        \sup_{\R^d}\|\nabla \calR'(\mu_t)\| \lesssim W_2^2(\mu_t,\muinf) \le 4\|f_t\|^2_{\dotH^{-1}(\muinf)}. 
   \end{equation}  
   Substituting this into the $\dotH^{-1}$ energy estimate in \cite[Section 3]{wang2026local}, we arrive at
   \begin{multline*}
        \frac{\dd}{\dd t}\|f_t\|^2_{\dotH^{-1}(\muinf)} \le -2(\beta^{-1}-\tau_0)\|f_t\|^2_{L^2(\muinf)}+  C\|f_t\|_{\dotH^{-1}(\muinf)} \Big[\|f_t\|_{L^2(\muinf)} \|f_t\|_{\dotH^{-1}(\muinf)} \\ + (1 + \|f_t\|_{L^2(\muinf)})\|f_t\|^2_{\dotH^{-1}(\muinf)}  \Big],
   \end{multline*}
   or equivalently, 
   \begin{equation}\label{eqn:Zt-At-mfld}
       \frac{\dd Z_t}{\dd t} \le -c_0 A_t +CZ_t\sqrt{A_t}+CZ_t^{3/2}(1+\sqrt{A_t}).
   \end{equation}  
   Similarly, the $L^2$ energy estimate now becomes
   \begin{equation}\label{eqn:At-Bt-mfld}
       \frac{\dd A_t}{\dd t} \le -2\beta^{-1} B_t +C\sqrt{B_tZ_t}(1+\sqrt{A_t})(1+\sqrt{Z_t}).
   \end{equation}
   As for the moment estimate, the extra term can be estimated using
   \begin{align*}
       & \Big|\int (1+f_t)\nabla \calR'(\mu_t)\cdot \nabla(\bangle{x}^\alpha f_t)\Big| \\ & \qquad \le \varepsilon
    \left(
        \int \bangle{x}^{\alpha}|\nabla f_t|^2\,\dd\nu_\infty
        +
        \int f_t^2\bangle{x}^{\alpha-2}\,\dd\nu_\infty
    \right)
    +
    C_\varepsilon Z_t^2
    \left(
        1+
        \int f_t^2\bangle{x}^{\alpha}\,\dd\nu_\infty
    \right).
   \end{align*} 
   this becomes
   \begin{equation}\label{eqn:energy-mfld}
        \frac{\dd}{\dd t}E_\alpha(t)
    \le
    C_0 A_t
    +
    C_0 Z_t(1+Z_t)\bigl(1+E_\alpha(t)\bigr).
   \end{equation}
   The rest of the arguments proceed similarly to \Cref{sec:wPIconv}. The inequality \eqref{eqn:Zt_At} does not depend on the dynamics, hence substituting it into \eqref{eqn:Zt-At-mfld}, and we derive
   \[\frac{\dd Z_t}{\dd t} \le -c_0A_t + C_R A_t(Z_t^{1-\frac{\alpha}{2(\alpha-2\sigma)}}+Z_t^{\frac{3}{2}-\frac{\alpha}{\alpha-2\sigma}}). \]
   As long as $\alpha>6\sigma$, we can choose $A_0$ (and hence $Z_0$) sufficiently small so that $\frac{\dd Z_t}{\dd t} \le -c_0A_t$ still holds, leading us to the same estimate as \eqref{eqn:Zt_decay}, and in particular, $Z_t$ is still decreasing. We comment here that when $F$ is quadratic, or $\calR=0$, we obtain identical differential inequalities as \eqref{eqn:wangchizatkey} and \eqref{eqn:Ealpha-diff-ineq}, and therefore $\alpha>4\sigma$ is sufficient. The rest of the proof, including the closure of bootstrap and the bound on $A_t$, is identical to that of \Cref{thm:decay_upper_bd}.
\end{proof}
   \begin{proof}[Proof of \Cref{thm:lowerbd-MFLD}]
       The proof is very similar to that of \Cref{thm:lowerbd}, and we only point out the differences. Here we need to define
       \[ \mathcal{A}f=\beta^{-1}Lf +LKf.\] As before, define $S(t)$ to be the contraction semigroup generated by $-\mathcal{A}$. The construction of $q_r$ is identical, and $\Psi_r$ is replaced by $q_r$. The estimate \eqref{eqn:full-quasimode-residual} and \eqref{eqn:quasimode-final} remain the same. The initial condition is chosen as $\kappa q_r$. The remainder term needs a separate estimate from \eqref{eqn:remainder-bound}:
       \begin{align*}
             \|\mathcal N(f_{r,s})\|_{\mathcal H}
    & = \Big(\|\nabla_x^*(f_{r,s}\nabla_x(Kf_{r,s}))\|^2_{\dotH^{-1}(\muinf)} + \big\|\nabla_x^*\big((1+f_{r,s})\nabla_x \calR'(\mu_s)\big)\big\|^2_{\dotH^{-1}(\muinf)}\Big)^\frac{1}{2} \\ & \le \Big(\|f_{r,s}\nabla_x(Kf_{r,s})\|^2_{L^2(\muinf)} + \|(1+f_{r,s})\nabla_x \calR'(\mu_s)\|^2_{L^2(\muinf)}\Big)^\frac{1}{2} \\ & \le \Big( \|f\|_{L^2(\muinf)}^2 \|\nabla_x (Kf_{r,s})\|_{L^\infty}^2 + (1+\|f_{r,s}\|_{L^2(\muinf)}^2)\|\nabla_x \calR'(\mu_s)\|_{L^\infty}^2\Big)^\frac{1}{2}
    \\ & \stackrel{\eqref{eqn:nablaKft},\eqref{eqn:LinfRpmu} }{\leq}
    C M\big(\sqrt{A_sZ_s}+(1+\sqrt{A_s})Z_s\big).
       \end{align*}
    The extra remainder term has the scaling (the term $\sqrt{A_s}Z_s$ is of higher order and thus neglected)
    \[\int _0^t Z_s \dd s \le tZ_0 \asymp \kappa^2 r^4 m_r, \]
    hence \[ \left|
        \langle w_r(t),q_r\rangle_{L^2(\mu_\infty)}
    \right|
    \leq
    \|w_r(t)\|_{\mathcal H}
    \|\nabla q_r\|_{L^2(\mu_\infty)} \lesssim \kappa^2 r^3m_r^{3/2} \ll \kappa m_r \asymp \kappa\|q_r\|^2_{L^2(\muinf)}. \]
    The last estimate uses $p>d_x+\alpha$ and $\alpha>6$. The rest of the proof proceeds identically.
   \end{proof}

\appendix

\section{Hypocoercivity}\label{sec:hyp}
% The main purpose of this section is to explore the hypocoercivity structure of the two-timescale mean-field GDA, and especially how the skew-symmetric operator transfers perturbations into directions dissipated by the symmetric operator. 
This appendix continues the finite-dimensional hypocoercivity analysis of GDA by An and Lu \cite{an2025convergence}. Their analysis of quadratic games uses modified Lyapunov functions to expose the interaction between symmetric dissipation and skew-symmetric coupling. We develop the corresponding formulation for two-timescale Wasserstein GDA under Assumption~\ref{ass:wpi}.

\Cref{thm:decay_upper_bd} provides the starting point: a cancellation in the negative Sobolev energy estimate, combined with weighted interpolation and a moment bootstrap, yields local algebraic stability for equal timescales without requiring a spectral gap. Here we retain this mechanism while making the operator structure explicit. After rescaling, the coupling is skew-symmetric in the homogeneous negative Sobolev space $\dot H^{-1}$. Resolvent-based cross terms then produce modified energies adapted to the two regimes $0<\eta\le 1$ and $\eta>1$. The additional quadratic forms in the resulting dissipation capture the contribution of the coupling, and the nonlinear error is controlled by the stability result of \Cref{thm:decay_upper_bd}.

Consider the two-timescale mean-field GDA 
 \begin{equation}\label{eqn:etaGDA}
     \begin{aligned}
         &\d_t \mu_t = \nabla_x \cdot\Big(\mu_t \int_{\cY} \nabla_x k(x,y) \dd\nu_t(y)\Big)+ \beta^{-1}\Delta_x \mu_t\\
         &\d_t \nu_t = \eta\Big(-\nabla_y \cdot\Big(\nu_t \int_{\cX} \nabla_yk(x,y) \dd\mu_t(x)   \Big)+ \beta^{-1}\Delta_y \nu_t\Big).
     \end{aligned}
 \end{equation}
With the same relative density perturbations $f_t=\dd\mu_t/\dd\mu_\infty-1$ and $g_t=\dd\nu_t/\dd\nu_\infty-1$ as in the main text, this becomes
\begin{equation}\label{eqn:etamodGDA}
\begin{aligned}
    \frac{\d}{\d t} f_t &= -\beta^{-1} L_x f_t - L_x Kg_t - \nabla_x^*(f_t\nabla_x K g_t),\\
    \frac{\d}{\d t} g_t &= \eta\Big(-\beta^{-1} L_y g_t + L_y K^\top f_t + \nabla_y^*(g_t\nabla_y K^\top f_t)\Big).
    \end{aligned}
\end{equation}
All perturbations have zero mean with respect to the corresponding equilibrium measure. Set $h_t=(\sqrt\eta f_t,g_t)^\top$. In the original time variable,
\begin{equation}
    \d_t h_t =
    -\widetilde{\mathcal L}_s h_t-\sqrt{\eta}\mathcal L_a h_t+\widetilde{\mathcal N}(h_t),
\end{equation}
where
\[
    \widetilde{\mathcal L}_s
    =\beta^{-1}
    \begin{pmatrix}
        L_x & 0\\
        0 & \eta L_y
    \end{pmatrix},
    \qquad
    \mathcal L_a
    =
    \begin{pmatrix}
        0 & L_xK\\
         -L_yK^\top & 0
    \end{pmatrix},
\]
and
\begin{equation}
        \widetilde{\mathcal N}(h_1, h_2) =
    \begin{pmatrix}
        -\nabla_x^*(h_1\nabla_xK h_2)\\
        \sqrt{\eta}\nabla_y^*(h_2\nabla_yK^\top h_1)
    \end{pmatrix}.
\end{equation}
Rescale time by $\tau=\sqrt\eta\,t$, and, with a slight abuse of notation,
write $h_\tau=(\sqrt\eta f_{\tau/\sqrt\eta},g_{\tau/\sqrt\eta})^\top$.
Then
\begin{equation}\label{eqn:dms_evol}
    \d_\tau h_\tau =
    -\mathcal L_s h_\tau-\mathcal L_a h_\tau+\mathcal N(h_\tau),
\end{equation}
with the same $\mathcal{L}_a$ and 
\begin{equation}
    \mathcal{L}_s
    =\beta^{-1}
    \begin{pmatrix}
        \eta^{-1/2} L_x & 0\\
        0 & \eta^{1/2} L_y
    \end{pmatrix},
    \qquad
    \mathcal N(h_1, h_2) =
    \begin{pmatrix}
        -\eta^{-1/2}\nabla_x^*(h_1\nabla_xK h_2)\\
        \nabla_y^*(h_2\nabla_yK^\top h_1)
    \end{pmatrix}.
\end{equation}
We work in the same homogeneous negative Sobolev space as in \Cref{sec:wPIconv}
\begin{equation*}
    \mathcal H
    =
    \dot H^{-1}(\mu_\infty)\oplus \dot H^{-1}(\nu_\infty)
\end{equation*}
with inner product
\begin{equation*} 
    \langle h, h'\rangle_{\mathcal H}
    =
    \langle h_1,L_x^{-1} h'_1\rangle_{L^2(\muinf)}
    +
    \langle h_2,L_y^{-1}h'_2\rangle_{L^2(\nuinf)}.
\end{equation*}
With respect to this inner product,
\begin{equation*}  
    \mathcal L_s^*=\mathcal L_s,
    \qquad
    \mathcal L_a^*=-\mathcal L_a,
\end{equation*}
which are symmetric and skew-symmetric operators, respectively.

For a fixed $\rho>0$, we define
\begin{equation}\label{eq:wc-resolvents}
\begin{aligned}
 \mathcal M_x&=L_y K^\top(\rho+\beta^{-1}L_x)^{-1}:\dot H^{-1}(\muinf)\longrightarrow \dot H^{-1}(\nuinf),\\
\mathcal M_y&=L_x K(\rho+\beta^{-1}L_y)^{-1}:\dot H^{-1}(\nuinf)\longrightarrow \dot H^{-1}(\muinf).
\end{aligned}
\end{equation}
Then, we introduce the modified Lyapunov energy
\begin{equation}\label{eqn:DMS-energy-def}
    \mathscr H^\varepsilon(h)
    :=
    \frac12\|h\|_{\mathcal H}^2+
    \begin{cases}
 \varepsilon_\eta \langle \mathcal{M}_x h_1, L_y^{-1}h_2\rangle_{L^2(\nuinf)},
       &0<\eta\le 1,\\
 -\eps_\eta\langle h_1, L_x^{-1}\mathcal{M}_y h_2\rangle_{L^2(\muinf)},
       &\eta>1.
\end{cases}
\end{equation}
where a sufficiently small $\eps_\eta>0$ is to be chosen depending on $\eta$. Note that with (\ref{eqn:bddk2}) in Assumption \ref{ass:wpi}, both $\mathcal{M}_x$ and $\mathcal{M}_y$ are bounded operators in the $\dot{H}^{-1}$ sense, therefore sufficiently small chosen $\eps_\eta$ ensures that there exists $0<\eps<1/2$,
\begin{equation}
    \frac{1-\eps}{2}\|h\|_{\mathcal H}^2\leq \mathscr H^\varepsilon(h)\leq \frac{1+\eps}{2}\|h\|_{\mathcal H}^2,
\end{equation}
and thus we have the norm equivalence. 
\begin{lemma}\label{lem:M_bounds}
    For $\muinf$ and $\nuinf$ satisfying Assumption \ref{ass:wpi} with $\alpha>4\sigma_*$, both $\mathcal{M}_x$ and $\mathcal{M}_y$ are bounded, in both $\dotH^{-1}$ and $L^2$ spaces, in the sense that for fixed $\rho>0$ and $M$ given by \eqref{eqn:bddk2}, for any mean-zero $f$ and $g$ in the corresponding $L^2\cap\dot H^{-1}$ spaces, there exist constants $M_{x,1},\, M_{x,2}, \, M_{y,1}, \, M_{y,2}>0$ such that
    \begin{align}
         \|\mathcal{M}_x f\|_{\dot{H}^{-1}(\nuinf)}\le\frac{M}{\rho}  \|f\|_{\dot{H}^{-1}(\muinf)}, & \qquad
\|\mathcal{M}_yg\|_{ \dot{H}^{-1}(\muinf)}\le\frac{M}{\rho} \|g\|_{\dot{H}^{-1}(\nuinf)}. \label{eqn:MxH-1} \\ \|\mathcal{M}_x f\|_{L^2(\nuinf)}^2\leq M_{x,1} \|f\|_{L^2(\muinf)}^2, &\qquad \|\mathcal{M}_y g\|_{L^2(\muinf)}^2\leq M_{y,1} \|g\|_{L^2(\nuinf)}^2, \label{eqn:MxL2}\\ |\langle \mathcal{M}_x f, K^\top f\rangle_{L^2(\nuinf)}|\leq M_{x,2}\|f\|_{L^2(\muinf)}^2,
         &\qquad |\langle \mathcal{M}_y g, K g\rangle_{L^2(\muinf)}|\leq M_{y,2}\|g\|_{L^2(\nuinf)}^2. \label{eqn:Mxmix}
    \end{align}
\end{lemma}
\begin{proof}
We first prove \eqref{eqn:MxH-1}. For $\psi_g=L_y^{-1}g$, weighted integration by parts gives
\begin{equation*}
 \nabla_x Kg(x)
 =\int_{\mathcal{Y}}\nabla_x\nabla_y k(x,y)\nabla_y\psi_g(y)\,
                        \mathrm d\nu_\infty(y).
\end{equation*}
Consequently,
$\|\nabla_xKg\|_{L^\infty}\le M\|g\|_{\dot H^{-1}(\nu_\infty)}$.
The same argument gives
$\|\nabla_yK^\top f\|_{L^\infty}\le M\|f\|_{\dot H^{-1}(\mu_\infty)}$.
These formulas extend by continuity to the homogeneous spaces.
The resolvent of either nonnegative self-adjoint generator has operator norm
at most $\rho^{-1}$ on its homogeneous negative Sobolev space. Therefore,
% \begin{equation*}
%  |\nabla_xKg(x)|
%  \le M\int|\nabla_y\psi_g|\,\mathrm d\nu_\infty
%  \le M\|g\|_{\dot{H}^{-1}(\nuinf)}.
% \end{equation*}
% Now for any function $\phi\in H^{-1}(\nuinf)$, we set $g= (\rho+\beta^{-1}L_y)^{-1} \phi$, then
\begin{equation*}
\|\mathcal{M}_y \phi\|_{\dot{H}^{-1}(\muinf)}= \|L_xK g\|_{\dot{H}^{-1}(\muinf)}=\|\nabla_x Kg\|_{L^2(\mu_\infty)}\le M\|g\|_{\dot{H}^{-1}(\nuinf)}\leq \frac{M}{\rho}\|\phi\|_{\dot{H}^{-1}(\nuinf)}.
\end{equation*}

We next prove \eqref{eqn:MxL2}. To prepare for this, we note that our Assumption \ref{ass:wpi} with $\alpha>4\sigma_*$ implies that $\muinf,\nuinf \in \mathcal{P}_2$. Indeed, if $\muinf$ has exponential tail, this is automatic; if $\muinf$ has algebraic tail, then $\sigma_*=1$ and $\alpha>4$, again $\muinf\in \mathcal{P}_2(\R^{d_x})$. We also note that $\muinf$ has finite Fisher information in the sense of 
\[ I(\muinf):=\int |\nabla \log \muinf|^2 \dd \muinf<\infty.\]
Indeed, the Hessian bound \eqref{eqn:bddk2} and equilibrium equation \eqref{eqn:equilibrium} means that $\nabla^2 \log \muinf$ must be bounded. Thus, let $\chi_R \in [0,1]$ be a cutoff function that is 1 in $B_R(0)$ and 0 in $B_{2R}^c(0)$ with $|\nabla \chi_R|\le \frac{2}{R}$, we have
\begin{align*}
    I_R&:= \int \chi_R^2|\nabla \log \muinf|^2 \dd \muinf  = -\int \chi_R^2 \Delta \log \muinf \dd \muinf -2\int \chi_R \nabla \chi_R\cdot \nabla \log \muinf \dd \muinf \\ & \le -\int \chi_R^2 \Delta \log \muinf \dd \muinf + \frac{1}{2}\int \chi_R^2|\nabla \log \muinf|^2 \dd \muinf + 2 \int |\nabla \chi_R|^2 \dd \muinf.
\end{align*}
The first term is uniformly bounded in $R$; the second term can be absorbed into left hand-side; while the third term converges to zero as $R\to\infty$. Hence $I_R$ is uniformly bounded in $R$, and Fatou's lemma allows us to pass to $R\to\infty$ to establish $I(\muinf)<\infty$. Similar estimates hold for $\nuinf$. Now take any $u\in L^2(\muinf)$ with mean zero, we have
\begin{align*}
    |\nabla_y K^\top u(y)|& = \Big|\int (\nabla_y k(x,y)-\nabla_y k(0,y)) u(x) \dd \muinf(x)\Big| \\ & \le M\int |x||u(x)| \dd \muinf \le M \sqrt{m_{2,x}}\|u\|_{L^2(\muinf)},
\end{align*} 
where $m_{2,x}:=\int |x|^2 \dd \muinf<\infty$. We similarly define $m_{2,y}:= \int |y|^2 \dd \nuinf<\infty$. Additionally, 
\[ |\Delta_y K^\top u(y)|\le M d_y \int |u(x)|\dd \muinf \le Md_y\|u\|_{L^2(\muinf)}.\]
Now take $u=(\rho+\beta^{-1}L_x)^{-1}f$, we arrive at
\begin{align*}
    \|\mathcal{M}_x f\|^2_{L^2(\nuinf)} & \le 2\Big(\|\Delta_y K^\top u\|_{L^2(\nuinf)}^2 + \|\nabla_y \log \nuinf\cdot \nabla_yK^\top u\|^2_{L^2(\nuinf)} \Big) \\ & \le 2\big(M^2 d_y^2+ M^2m_{2,x}I(\nuinf)\big)\|u\|^2_{L^2(\muinf)} \le \frac{2}{\rho}\big(M^2 d_y^2+ M^2m_{2,x}I(\nuinf)\big)\|f\|^2_{L^2(\muinf)}.
\end{align*}
This is \eqref{eqn:MxL2}.

We finally treat the mixed term in \eqref{eqn:Mxmix}. By Cauchy-Schwarz inequality, and the fact that $\mathcal{M}_x f$ has zero mean, it suffices to prove
\begin{equation}\label{eqn:KTfL2}
    \Big\|K^\top f - \int K^\top f \dd \nuinf \Big\|_{L^2(\nuinf)} \lesssim \|f\|_{L^2(\muinf)}.
\end{equation} 
The trick is to rewrite the expression using centrality of $f$, that is,
\begin{align*}
    \Big|K^\top f(y)-\int K^\top f \dd \nuinf\Big| & = \Big| \int \big(k(x,y)-k(x,z)\big) f(x)\dd\muinf(x) \dd \nuinf(z)\Big| \\ & = \Big| \int \big(k(x,y)-k(x,z)-k(0,y)+k(0,z)\big) f(x)\dd\muinf(x) \dd \nuinf(z)\Big| \\ & \le M\int |x||f(x)|\dd \muinf(x) \int |y-z|\dd \nuinf(z) \\ & \le M\sqrt{m_{2,x}}\|f\|_{L^2(\muinf)}\sqrt{2(|y|^2 + m_{2,y})}.
\end{align*}
This implies
\[ \int \Big|K^\top f(y)-\int K^\top f \dd \nuinf\Big|^2 \dd \nuinf(y) \le 4M^2m_{2,x}m_{2,y}\|f\|^2_{L^2(\muinf)},\] which is \eqref{eqn:KTfL2}. 
The estimates for $\mathcal M_y$ are identical.
\end{proof}

%We need further estimates on $\mathcal{M}_x$ and $\mathcal{M}_y$, which are unfortunately not implied by Lemma \ref{lem:M_bounds}.
%\begin{assumption}\label{ass:l2andcbounds}
%    There exists constants $M_{x,1}, M_{x,2}, M_{y,1}, M_{y,2}>0$ such that for all $f\in H^{-1}(\muinf),~ g\in H^{-1}(\nuinf)$,
    
%\end{assumption}
\begin{thm}\label{thm:eta-stability}
   Suppose Assumption~\ref{ass:wpi} holds with
$\alpha>4\sigma_*$. Fix $\eta,\rho>0$ and choose sufficiently small $R_0>0$ independent of $\eta$, and let $h_\tau$ solve
\eqref{eqn:dms_evol} with initial condition $h_0$ corresponding to
probability densities. There exists $\epsilon_*>0$, depending on
$\eta,\rho,R_0>0$ and the constants in the assumptions, such that if initial data $h_0$ satisfies
\[
\eta^{-1}\int h_{0,1}^2\langle x\rangle^\alpha\,d\mu_\infty
       +\int h_{0,2}^2\langle y\rangle^\alpha\,d\nu_\infty<\frac {R_0}{2},
\qquad
\eta^{-1}\|h_{0,1}\|_{L^2(\mu_\infty)}^2
       +\|h_{0,2}\|_{L^2(\nu_\infty)}^2\le\epsilon_*,
\]
    then we have the convergence (the constant $C>0$ depending on $\beta, R, c_{\rm WP}^x,c_{\rm WP}^y$ but not on $\eta$),
    \begin{equation*}
    \|h_\tau\|_{\mathcal{H}}^2\lesssim \left[\|h_0\|_{\mathcal{H}}^{-\frac{4\sigma_*}{\alpha-2\sigma_*}}+\frac{C\min\{\eta^{1/2},\eta^{-1/2}\} \tau}{\max\{\eta,1\}^{\frac{2\sigma_*}{\alpha-2\sigma_*}}}\right]^{-\frac{\alpha-2\sigma_*}{2\sigma_*}}.
    \end{equation*}
\end{thm}
We prove Theorem~\ref{thm:eta-stability} in four steps. Sections \ref{appendix:small} and
\ref{appendix:large} estimate the linear terms in the two modified energies, \ref{appendix:error} controls
the nonlinear remainders, and \ref{appendix:convergence} closes the
argument.
\subsection{The small $\eta$ regime}\label{appendix:small}
We consider $0<\eta\leq 1$ and write $\eps = \eps_\eta,~ L^2(\nuinf) = L^2$ temporarily for notation simplicity. By differentiating  the modified energy combined with (\ref{eqn:dms_evol}), we have 
\begin{equation}
\begin{aligned}\label{eqn:dms_der1}
    \frac{d}{d\tau}\mathscr H^\varepsilon(h_\tau)
    &=
    -\langle \mathcal{L}_s h_\tau,h_\tau\rangle_{\mathcal H}
    + \mathcal R_x^\eps(h_\tau) -\frac{\eps}{\beta \sqrt{\eta}}
    \langle \mathcal{M}_x L_x h_{\tau,1},L_y^{-1}h_{\tau,2}\rangle_{L^2}
    \\
    &\quad
    -
    \eps
    \langle \mathcal{M}_x L_x K h_{\tau,2}, L_y^{-1}h_{\tau,2}\rangle_{L^2}
    -
    \frac{\eps\sqrt{\eta}}{\beta}
    \langle \mathcal{M}_x h_{\tau,1},h_{\tau,2}\rangle_{L^2} + \eps\langle \mathcal{M}_x h_{\tau,1},K^\top h_{\tau,1}\rangle_{L^2},
\end{aligned}
\end{equation}
where $\mathcal{R}_x^\eps$ is a nonlinear error 
\begin{equation}\label{eqn:error1}
    \mathcal R_x^\varepsilon(h)
    =
    \langle \mathcal N(h),h\rangle_{\mathcal H}
    +
    \varepsilon\langle \mathcal{M}_x\mathcal N(h)_1, L_y^{-1 }h_2\rangle_{L^2}
    +
    \varepsilon\langle \mathcal{M}_x h_1,L_y^{-1}\mathcal{N}_\eta(h)_2\rangle_{L^2}.
\end{equation}
% Writing $\psi_f=L_x^{-1}f$, the adjoint identity follows from
% \begin{equation*}
% \begin{aligned}
% \langle L_x K g,L_x^{-1}f\rangle_{L^2(\muinf)}
% &=\int_{\mathcal{Y}}\int_{\mathcal{X}} \psi_f(x)^{\top}
%        \nabla_x\nabla_y k(x,y)\nabla\psi_g(y)\,
%                       \mathrm d\mu_\infty(x)\mathrm d\nu_\infty(y)\\
% &=\langle g,L_y K^\top f\rangle_{H_y}.
% \end{aligned}
% \end{equation*}
% The double integral is absolutely convergent by (A2) and
% Cauchy--Schwarz. These formulas justify the extensions without requiring
% $Kg$ itself to belong to $L^2(\mu_\infty)$.

% \begin{equation}\label{eqn:DMS-energy-equiv}
%     \frac{1-\varepsilon}{2}\|h\|_{\mathcal H}^2
%     \le
%     \mathscr H_\Pi^\varepsilon(h)
%     \le
%     \frac{1+\varepsilon}{2}\|h\|_{\mathcal H}^2.
% \end{equation}
\noindent The dissipation part is
\begin{equation}
    \langle \mathcal{L}_s h_\tau,h_\tau\rangle_{\mathcal H} = \frac{1}{\beta}\Big(\eta^{-1/2}\|h_{\tau,1}\|_{L^2(\muinf)}^2+\eta^{1/2}\|h_{\tau,2}\|_{L^2(\nuinf)}^2\Big).
\end{equation}
We leave $R_x^\eps$ to be estimated later and now focus on the last four terms in (\ref{eqn:dms_der1}). First of all, spectral calculus gives $(\rho+\beta^{-1}L_x)^{-2}\leq \beta L_x^{-1} (\rho+\beta^{-1}L_x)^{-1}$, and hence
\begin{equation*}
\begin{aligned}
    \Big| 
    \langle \mathcal{M}_x L_x h_{\tau,1},L_y^{-1}h_{\tau,2}\rangle_{L^2}\Big|&=\Big| 
    \langle  h_{\tau,1},(\rho+\beta^{-1}L_x)^{-1}L_x K h_{\tau,2}\rangle_{L^2}\Big|\\
    &\leq \sqrt{\beta}\|h_{\tau,1}\|_{L^2(\muinf)}\|(\rho+\beta^{-1}L_x)^{-1/2} L_x K h_{\tau,2}\|_{\dot{H}^{-1}(\muinf)}\\
    &\leq \frac{\sqrt{\beta}}{4\gamma}\|h_{\tau,1}\|_{L^2(\muinf)}^2+ \gamma\sqrt{\beta}\|(\rho+\beta^{-1}L_x)^{-1/2} L_x K h_{\tau,2}\|_{\dot{H}^{-1}(\muinf)}^2,
\end{aligned}
\end{equation*}
with $\gamma>0$ to be chosen later. Secondly, 
\begin{equation*}
    -\eps \langle \mathcal{M}_x L_x K h_{\tau,2}, L_y^{-1}h_{\tau,2}\rangle_{L^2} = -\eps\|(\rho+\beta^{-1}L_x)^{-1/2} L_x K h_{\tau,2}\|_{\dot{H}^{-1}(\muinf)}^2.
\end{equation*}
Moreover, by \eqref{eqn:MxL2} and \eqref{eqn:Mxmix} of \Cref{lem:M_bounds}, we have
\begin{equation*}
\begin{aligned}
    \Big|-
    \frac{\eps\sqrt{\eta}}{\beta}
    \langle \mathcal{M}_x h_{\tau,1},h_{\tau,2}\rangle_{L^2} \Big|&\leq \frac{\eps\sqrt{\eta}}{\beta}\|\mathcal{M}_x h_{\tau,1}\|_{L^2(\nuinf)} \|h_{\tau,2}\|_{L^2(\nuinf)}\\
    & \leq \frac{\eps\sqrt{\eta}M_{x,1}}{2\beta}\| h_{\tau,1}\|_{L^2(\muinf)}^2+  \frac{\eps\sqrt{\eta}}{2\beta} \|h_{\tau,2}\|_{L^2(\nuinf)}^2
    \end{aligned}
\end{equation*}
and 
\begin{equation*}
    \Big| \eps\langle \mathcal{M}_x h_{\tau,1},K^\top h_{\tau,1}\rangle_{L^2}\Big|\leq \eps M_{x,2} \|h_{\tau,1}\|_{L^2(\muinf)}^2.
\end{equation*}
Hence, collecting all bounds above together and rearrange, we obtain that
\begin{equation}\label{app:eqnsmall}
\begin{aligned}
\frac{d}{d\tau}\mathscr H^\varepsilon(h_\tau)&\leq -\eta^{-1/2}\Big(\frac{1}{\beta}-\frac{\eps}{4\gamma\sqrt{\beta}}-\frac{\eps\eta M_{x,1}}{2\beta}-\eps\sqrt{\eta}M_{x,2}\Big) \| h_{\tau,1}\|_{L^2(\muinf)}^2\\
&-\frac{\sqrt{\eta}}{\beta}\big(1-\frac{\eps}{2}\big)\|h_{\tau,2}\|_{L^2(\nuinf)}^2+\eps \big(\frac{\gamma}{\sqrt{\eta \beta}}-1\big)\|(\rho+\beta^{-1}L_x)^{-1/2} L_x K h_{\tau,2}\|_{\dot{H}^{-1}(\muinf)}^2\\
&+ \mathcal R_x^\varepsilon(h_\tau).
\end{aligned}
\end{equation}
Since $0<\eta\leq 1$, we may choose $\gamma=\sqrt{\eta\beta}/2$, and set sufficiently small $\eps_\eta>0$, so that it satisfies 
\begin{equation*}
    \eps_\eta\leq 1/2, \quad\text{and }\quad \eps_\eta \Big(\frac{1}{2\sqrt{\eta}}+\frac{\eta M_{x,1}}{2}+\beta \sqrt{\eta} M_{x,2}\Big)\leq 1/4
\end{equation*}
then the third term in (\ref{app:eqnsmall}) can be dropped, and we get
\begin{equation}
    \frac{d}{d\tau}\mathscr H^\varepsilon(h_\tau)\leq -\frac{3\sqrt{\eta}}{4\beta} \|h_\tau\|_{L^2}^2 + \mathcal R_x^\varepsilon(h_\tau).
\end{equation}
Here and below, we have $\|h\|_{L^2}^2 = \|h_1\|_{L^2{\muinf}}^2+\|h_2\|_{L^2{\nuinf}}^2$.
\subsection{The large $\eta$ regime}\label{appendix:large}
We consider $\eta>1$ and  write $\eps = \eps_\eta,~ L^2(\muinf) = L^2$ temporarily for notation simplicity. Similar as the previous section, but instead we consider $\mathscr H^\eps(h)=\frac{1}{2}\|h\|_{\mathcal{H}}^2-\eps\langle h_1, L_x^{-1}\mathcal{M}_y h_2\rangle_{L^2}$ for large $\eta$, differentiation then gives 
\begin{equation}
\begin{aligned}\label{eqn:dms_der2}
    \frac{d}{d\tau}\mathscr H^\varepsilon(h_\tau)
    &=
    -\langle \mathcal{L}_s h_\tau,h_\tau\rangle_{\mathcal H}
    + \mathcal R_y^\eps(h_\tau) +\frac{\eps}{\beta \sqrt{\eta}}
    \langle  L_x h_{\tau,1},L_x^{-1} \mathcal{M}_y h_{\tau,2}\rangle_{L^2}
    + \eps
    \langle  L_x K h_{\tau,2}, L_x^{-1}\mathcal{M}_y h_{\tau,2}\rangle_{L^2}\\
    &\quad 
    +\frac{\eps\sqrt{\eta}}{\beta}
    \langle  h_{\tau,1},L_x^{-1}\mathcal{M}_y L_y h_{\tau,2}\rangle_{L^2} - \eps\langle  h_{\tau,1},L_x^{-1}\mathcal{M}_y L_y K^\top h_{\tau,1}\rangle_{L^2},
\end{aligned}
\end{equation}
where 
\begin{equation}\label{eqn:error2}
    \mathcal R_y^\varepsilon(h)
    =
    \langle \mathcal N(h),h\rangle_{\mathcal H}
    -
    \varepsilon\langle \mathcal N(h)_1, L_x^{-1 }\mathcal{M}_y h_2\rangle_{L^2}
    -
    \varepsilon\langle  h_1,L_x^{-1}\mathcal{M}_y\mathcal{N}_\eta(h)_2\rangle_{L^2}.
\end{equation}
Here, the last four terms in (\ref{eqn:dms_der2}) can treated similarly as in the previous subsection:
\begin{equation*}
\begin{aligned}
&\Big|\frac{\eps}{\beta \sqrt{\eta}}
    \langle  L_x h_{\tau,1},L_x^{-1} \mathcal{M}_y h_{\tau,2}\rangle_{L^2}\Big|\leq  \frac{\eps}{2\beta\sqrt{\eta}}\| h_{\tau,1}\|_{L^2(\muinf)}^2+  \frac{\eps M_{y,1}}{2\beta\sqrt{\eta}} \|h_{\tau,2}\|_{L^2(\nuinf)}^2,\\
&\Big|\eps
    \langle  L_x K h_{\tau,2}, L_x^{-1}\mathcal{M}_y h_{\tau,2}\rangle_{L^2}\Big|\leq \eps M_{y,2}\|h_{\tau,2}\|_{L^2(\nuinf)}^2,\\
    &\Big|
    \langle  h_{\tau,1},L_x^{-1}\mathcal{M}_y L_y h_{\tau,2}\rangle_{L^2} \Big|\leq \frac{\sqrt{\beta}}{4\gamma}\|h_{\tau,2}\|_{L^2(\nuinf)}^2+ \gamma\sqrt{\beta}\|(\rho+\beta^{-1}L_y)^{-1/2} L_y K^\top h_{\tau,1}\|_{\dot{H}^{-1}(\nuinf)}^2,\\
    &- \eps\langle  h_{\tau,1},L_x^{-1}\mathcal{M}_y L_y K^\top h_{\tau,1}\rangle_{L^2} = -\eps\|(\rho+\beta^{-1}L_y)^{-1/2} L_y K^\top h_{\tau,1}\|_{\dot{H}^{-1}(\nuinf)}^2.
\end{aligned}
\end{equation*}
Hence, with all bounds above, we have 
\begin{equation}
\begin{aligned}
\frac{d}{d\tau}\mathscr H^\varepsilon(h_\tau)&\leq -\sqrt{\eta}\Big(\frac{1}{\beta}-\frac{\eps}{4\gamma\sqrt{\beta}}-\frac{\eps M_{y,1}}{2\beta \eta }-\frac{\eps M_{y,2}}{\sqrt{\eta}}\Big) \| h_{\tau,2}\|_{L^2(\nuinf)}^2\\
&-\frac{1}{\beta\sqrt{\eta}}\big(1-\frac{\eps}{2}\big)\|h_{\tau,1}\|_{L^2(\muinf)}^2+\eps \big(\frac{\gamma\sqrt{\eta}}{\sqrt{\beta}}-1\big)\|(\rho+\beta^{-1}L_y)^{-1/2} L_y K^\top h_{\tau,1}\|_{\dot{H}^{-1}(\nuinf)}^2\\
&+ \mathcal R_y^\eps(h_\tau).
\end{aligned}
\end{equation}
Now we set $\gamma = \frac{\sqrt{\beta}}{2\sqrt{\eta}}$, and choose sufficiently small $\eps_\eta>0$ such that
\begin{equation*}
    \eps_\eta\leq 1/2, \quad\text{and }\quad \eps_\eta \Big(\frac{\sqrt{\eta}}{2}+\frac{ M_{y,1}}{2\eta}+\frac{\beta  M_{y,2}}{\sqrt{\eta}}\Big)\leq 1/4,
\end{equation*}
then we can conclude that
\begin{equation}
    \frac{d}{d\tau}\mathscr H^\varepsilon(h_\tau)\leq -\frac{3}{4\beta\sqrt{\eta}} \|h_\tau\|_{L^2}^2 + \mathcal R_y^\varepsilon(h_\tau).
\end{equation}
\subsection{Control the error $\mathcal{R}^\varepsilon$}\label{appendix:error}
We know that with (\ref{eqn:bddk2}) in Assumption \ref{ass:wpi}, 
\begin{equation*}
\begin{aligned}
    \|\mathcal{N}(h)_1\|_{\dot{H}^{-1}(\muinf)}&\leq \frac{M}{\sqrt{\eta}}
\|h_1\|_{L^2(\muinf)}\|h_2\|_{\dot{H}^{-1}(\nuinf)},\\
    \|\mathcal{N}(h)_2\|_{\dot{H}^{-1}(\nuinf)}&\leq M
    \|h_2\|_{L^2(\nuinf)}\|h_1\|_{\dot{H}^{-1}(\muinf)}.
  \end{aligned}
\end{equation*}
When $0<\eta\leq 1$, Lemma \ref{lem:M_bounds} gives that,
\begin{equation*}
    \begin{aligned}
        |R_x^\eps(h)|&\leq |\langle \mathcal N(h),h\rangle_{\mathcal H}|
    +
    \eps_\eta|\langle \mathcal{M}_x\mathcal N(h)_1, L_y^{-1 }h_2\rangle_{L^2}|
    +
    \eps_\eta|\langle \mathcal{M}_x h_1,L_y^{-1}\mathcal{N}_\eta(h)_2\rangle_{L^2}|\\
    &\leq M \|h_1\|_{\dot{H}^{-1}(\muinf)}\|h_2\|_{\dot{H}^{-1}(\nuinf)}\big(\eta^{-1/2}\|h_1\|_{L^2(\muinf)}+\|h_2\|_{L^2(\nuinf)}\big)\\
    &\quad +\frac{\eps_\eta M^2}{\rho \sqrt{\eta}}\|h_1\|_{L^2(\muinf)}\|h_2\|_{\dot{H}^{-1}(\nuinf)}^2 + \frac{\eps_\eta M^2}{\rho} \|h_2\|_{L^2(\nuinf)}\|h_1\|_{\dot{H}^{-1}(\muinf)}^2.
    \end{aligned}
\end{equation*}
And for $\eta>1$, we also have the same uppper bound since
\begin{align*}
\eps_\eta|\langle \mathcal N(h)_1,L_x^{-1}\mathcal M_yh_2\rangle_{L^2(\mu_\infty)}|
&\le\frac{\eps_\eta M}\rho
       \|\mathcal N(h)_1\|_{\dot H^{-1}(\mu_\infty)}
       \|h_2\|_{\dot H^{-1}(\nu_\infty)}\\
&\le\frac{\eps_\eta M^2}{\rho\sqrt\eta}
       \|h_1\|_{L^2(\mu_\infty)}
       \|h_2\|_{\dot H^{-1}(\nu_\infty)}^2,\\
\eps_\eta|\langle h_1,L_x^{-1}\mathcal M_y\mathcal N(h)_2\rangle_{L^2(\mu_\infty)}|
&\le\frac{\eps_\eta M}\rho
       \|h_1\|_{\dot H^{-1}(\mu_\infty)}
       \|\mathcal N(h)_2\|_{\dot H^{-1}(\nu_\infty)}\\
&\le\frac{\eps_\eta M^2}\rho
       \|h_2\|_{L^2(\nu_\infty)}
       \|h_1\|_{\dot H^{-1}(\mu_\infty)}^2.
\end{align*}
In particular, the different signs in the two definitions of the modified
energy do not affect the absolute remainder estimates.

To summarize, we utilize the same notation as in the proof of Theorem \ref{thm:decay_upper_bd} and set $Z=\|h\|_{\mathcal H}^2, A=\|h_1\|_{L^2(\muinf)}^2+\|h_2\|_{L^2(\nuinf)}^2$.
By Cauchy-Schwarz inequality, 
\[
\|h_1\|_{\dot H^{-1}(\mu_\infty)}
\|h_2\|_{\dot H^{-1}(\nu_\infty)}\le\frac Z2,\qquad
\eta^{-1/2}\|h_1\|_{L^2(\mu_\infty)}+\|h_2\|_{L^2(\nu_\infty)}
 \le\sqrt{1+\eta^{-1}}\sqrt A.
\]
Hence, both error terms obey
\begin{equation}\label{eqn:error-summary}
|\mathcal R_x^\varepsilon(h_{\tau})|,\ |\mathcal R_y^\varepsilon(h_{\tau})|
 \leq \left(\frac M2+\frac{\eps_\eta M^2}\rho\right)
          \sqrt{1+\eta^{-1}}\,Z\sqrt A
 =:C_\eta Z\sqrt A.
\end{equation}
Together with Sections \ref{appendix:small} and \ref{appendix:large}, this gives, in either regime,
\begin{equation}\label{eq:app-master-energy}
\frac d{d\tau}\mathscr H^\varepsilon(h_\tau)
 \le-a_\eta \|h_\tau\|_{L^2}^2+C_\eta Z_\tau\sqrt{A_\tau},\qquad
a_\eta:=\frac{3}{4\beta}\min\{\eta^{1/2},\eta^{-1/2}\}.
\end{equation}
Here $Z_\tau=\|h_\tau\|_{\mathcal H}^2$ and $A_\tau=\|h_\tau\|_{L^2}^2$.
Smallness of $Z_\tau$ alone does not directly absorb the last term, since
there is no spectral gap. The weighted interpolation established next is
the required additional ingredient.

\subsection{Algebraic convergence}\label{appendix:convergence}
Due the rescaling, let us also introduce
\begin{equation*}
    E_{\alpha}^\eta(\tau):=\eta^{-1}\int h_{\tau,1}^2 \bangle{x}^\alpha \dd\mu_\infty(x) +\int h_{\tau,2}^2 \bangle{y}^\alpha \dd\nu_\infty(y),
\end{equation*}
and $A_\tau \leq \max\{\eta,1\} E_{\alpha}^\eta(\tau)$. 
With weighted Poincar\'e inequalities (\ref{eqn:cwpx}), (\ref{eqn:cwpy}) in Assumption \ref{ass:wpi}, we have
\begin{equation*}
    \|h_{\tau,1}\|_{\dot H^{-1}(\mu_\infty)}^2
\le(c_{\rm WP}^x)^{-1}
 \left(\int h_{\tau,1}^2\langle x\rangle^\alpha\,d\mu_\infty\right)^{\frac{2\sigma_*}{\alpha}}
 \|h_{\tau,1}\|_{L^2(\mu_\infty)}^{2(1-2\sigma_*/\alpha)}.
\end{equation*}
Fix $R_0>0$ and suppose $E_\alpha^\eta(0)\le R_0/2$. Then there exists $\epsilon_*>0$ such that, if
\[
\eta^{-1}\|h_{0,1}\|_{L^2(\muinf)}^2+\|h_{0,2}\|_{L^2(\nuinf)}^2\le\epsilon_*,
\]
\Cref{thm:decay_upper_bd} gives $E_\alpha^\eta(\tau)\le R_0$ for all $\tau\ge0$. Consequently,
\begin{equation}
    Z_\tau\le\frac{\big(\max\{\eta,1\}R_0\big)^{\frac{2\sigma_*}{\alpha}}}{\min\{\cwp^x,\cwp^y\}}\,A_\tau^{1-\frac{2\sigma_*}{\alpha}},
\end{equation}
and therefore
\begin{equation}
    C_\eta Z_\tau\sqrt{A_\tau}\le\frac{C_\eta\big(\max\{\eta,1\}R_0\big)^{\frac{2\sigma_*}{\alpha}}}{\min\{\cwp^x,\cwp^y\}}\,A_\tau^{\frac32-\frac{2\sigma_*}{\alpha}}.
\end{equation}
Since $\alpha>4\sigma_*$, the power of $A_\tau$ above is strictly larger than $1$. Using $A_\tau\le \max\{\eta,1\}R_0$, we can choose $R_0$ sufficiently small, independently of $\eta$, because
\[
C_\eta\big(\max\{\eta,1\}R_0\big)^{\frac{2\sigma_*}{\alpha}}A_\tau^{\frac12-\frac{2\sigma_*}{\alpha}}
\lesssim M\sqrt{R_0}\,\max\{\eta,1\}^{1/2}\sqrt{1+\eta^{-1}}
\lesssim M\sqrt{R_0}.
\]
This yields
\begin{equation}
    \frac{\dd}{\dd\tau}\mathscr H^\varepsilon(h_\tau)
    \le -\frac{a_\eta}{2}\|h_\tau\|_{L^2}^2
    \le -\frac{a_\eta}{2}\,\frac{\big(\min\{\cwp^x,\cwp^y\}\big)^{\frac{\alpha}{\alpha-2\sigma_*}}}{\big(\max\{\eta,1\}R_0\big)^{\frac{2\sigma_*}{\alpha-2\sigma_*}}}\,Z_\tau^{\frac{\alpha}{\alpha-2\sigma_*}}.
\end{equation}
By the norm equivalence between $\mathscr H^\varepsilon(h_\tau)$ and $Z_\tau$, we can find a constant $C>0$ depending on $\beta, R_0, c_{\rm WP}^x,c_{\rm WP}^y$ such that 
\begin{equation}
    \mathscr H^\varepsilon(h_\tau)\leq \left[\mathscr H^\varepsilon(h_0)^{-\frac{2\sigma_*}{\alpha-2\sigma_*}}+\frac{C\min\{\eta^{1/2},\eta^{-1/2}\} \tau}{\max\{\eta,1\}^{\frac{2\sigma_*}{\alpha-2\sigma_*}}}\right]^{-\frac{\alpha-2\sigma_*}{2\sigma_*}},
\end{equation}
which implies that
\begin{equation}
    \|h_\tau\|_{\mathcal{H}}^2\lesssim \left[\|h_0\|_{\mathcal{H}}^{-\frac{4\sigma_*}{\alpha-2\sigma_*}}+\frac{C\min\{\eta^{1/2},\eta^{-1/2}\} \tau}{\max\{\eta,1\}^{\frac{2\sigma_*}{\alpha-2\sigma_*}}}\right]^{-\frac{\alpha-2\sigma_*}{2\sigma_*}}
\end{equation} \hfill $\Box$

\bibliographystyle{abbrv} 
\bibliography{references}

\begin{thebibliography}{10}

\bibitem{ambrosio2008gradient}
L.~Ambrosio, N.~Gigli, and G.~Savar{\'e}.
\newblock {\em Gradient flows in metric spaces and in the space of probability
  measures}.
\newblock Springer Science \& Business Media, 2008.

\bibitem{an2025convergence}
J.~An and J.~Lu.
\newblock Convergence of two-timescale gradient descent ascent dynamics:
  finite-dimensional and mean-field perspectives.
\newblock {\em arXiv preprint arXiv:2501.17122}, 2025.

\bibitem{bakry_analysis_2014}
D.~Bakry, I.~Gentil, and M.~Ledoux.
\newblock {\em Analysis and geometry of {Markov} diffusion operators}.
\newblock Springer, Cham; New York, 2014.

\bibitem{blanchet2007hardy}
A.~Blanchet, M.~Bonforte, J.~Dolbeault, G.~Grillo, and J.-L. V{\'a}zquez.
\newblock {H}ardy--{P}oincar{\'e} inequalities and applications to nonlinear
  diffusions.
\newblock {\em Comptes rendus. Math{\'e}matique}, 344(7):431--436, 2007.

\bibitem{bobkov2009weighted}
S.~G. Bobkov and M.~Ledoux.
\newblock Weighted {P}oincar{\'e}-type inequalities for {C}auchy and other
  convex measures.
\newblock {\em The Annals of Probability}, 37(2):403--427, 2009.

\bibitem{bouin2021hypocoercivity}
E.~Bouin, J.~Dolbeault, L.~Lafleche, and C.~Schmeiser.
\newblock Hypocoercivity and sub-exponential local equilibria.
\newblock {\em Monatshefte f{\"u}r Mathematik}, 194(1):41--65, 2021.

\bibitem{bouin2023mathrm}
E.~Bouin, J.~Dolbeault, and L.~Ziviani.
\newblock ${L}^2$ hypocoercivity methods for kinetic {F}okker-{P}lanck
  equations with factorised {G}ibbs states.
\newblock In {\em INdAM Meeting: Kolmogorov Operators and their Applications
  Workshop}, pages 23--56. Springer, 2022.

\bibitem{brevsar2026subexponential}
M.~Bre{\v{s}}ar and A.~Mijatovi{\'c}.
\newblock Subexponential lower bounds for f-ergodic {M}arkov processes.
\newblock {\em Probability Theory and Related Fields}, 194(1):997--1054, 2026.

\bibitem{brigati2024explicit}
G.~Brigati, G.~Stoltz, A.~Q. Wang, and L.~Wang.
\newblock Explicit convergence rates of underdamped langevin dynamics under
  weighted and weak poincar$\backslash$'e--lions inequalities.
\newblock {\em arXiv preprint arXiv:2407.16033, to appear in Annals of Applied
  Probability}, 2024.

\bibitem{busoniu2008comprehensive}
L.~Busoniu, R.~Babuska, and B.~De~Schutter.
\newblock A comprehensive survey of multiagent reinforcement learning.
\newblock {\em IEEE Transactions on Systems, Man, and Cybernetics, Part C
  (Applications and Reviews)}, 38(2):156--172, 2008.

\bibitem{cai2024convergence}
Y.~Cai, S.~Mitra, X.~Wang, and A.~Wibisono.
\newblock On the convergence of min-max {L}angevin dynamics and algorithm.
\newblock In {\em The Thirty Eighth Annual Conference on Learning Theory}.
  PMLR, 2025.

\bibitem{cao2020kinetic}
C.~Cao.
\newblock The kinetic {F}okker--{P}lanck equation with weak confinement force.
\newblock {\em Communications in Mathematical Sciences}, 17(8):2281--2308,
  2019.

\bibitem{carrillo2020long}
J.~A. Carrillo, R.~S. Gvalani, G.~A. Pavliotis, and A.~Schlichting.
\newblock Long-time behaviour and phase transitions for the {Mckean–Vlasov}
  equation on the torus.
\newblock {\em Archive for Rational Mechanics and Analysis}, 235:635–690,
  2020.

\bibitem{cattiaux2010functional}
P.~Cattiaux, N.~Gozlan, A.~Guillin, and C.~Roberto.
\newblock Functional inequalities for heavy tailed distributions and
  application to isoperimetry.
\newblock {\em Electronic Journal of Probability}, 15:346--385, 2010.

\bibitem{chen2025uniform}
F.~Chen, Z.~Ren, and S.~Wang.
\newblock {Uniform-in-time propagation of chaos for mean field Langevin
  dynamics}.
\newblock {\em Annales de l'Institut Henri Poincaré, Probabilités et
  Statistiques}, 61(4):2357 -- 2404, 2025.

\bibitem{chewi2026uniform}
S.~Chewi, A.~Nitanda, and M.~S. Zhang.
\newblock Uniform-in-\({N}\) log-sobolev inequality for the mean-field langevin
  dynamics with convex energy.
\newblock {\em SIAM Journal on Mathematical Analysis}, 58(4):3652--3668, 2026.

\bibitem{chizat2022meanfield}
L.~Chizat.
\newblock Mean-field {L}angevin dynamics: Exponential convergence and
  annealing.
\newblock {\em Transactions on Machine Learning Research}, 2022.

\bibitem{chizat2018global}
L.~Chizat and F.~Bach.
\newblock On the global convergence of gradient descent for over-parameterized
  models using optimal transport.
\newblock {\em Advances in neural information processing systems}, 31, 2018.

\bibitem{conforti2023game}
G.~Conforti, A.~Kazeykina, and Z.~Ren.
\newblock Game on random environment, mean-field langevin system, and neural
  networks.
\newblock {\em Mathematics of Operations Research}, 48(1):78--99, 2023.

\bibitem{conger2024coupled}
L.~Conger, F.~Hoffmann, E.~Mazumdar, and L.~J. Ratliff.
\newblock Coupled {W}asserstein gradient flows for min-max and cooperative
  games.
\newblock {\em arXiv preprint arXiv:2411.07403}, 2024.

\bibitem{conger2025monotone}
L.~Conger, F.~Hoffmann, E.~Mazumdar, and L.~J. Ratliff.
\newblock Monotone multispecies flows.
\newblock {\em arXiv preprint arXiv:2506.22947}, 2025.

\bibitem{cormier2025stability}
Q.~Cormier.
\newblock {On the stability of the invariant probability measures of
  McKean–Vlasov equations}.
\newblock {\em Annales de l'Institut Henri Poincaré, Probabilités et
  Statistiques}, 61(4):2405 -- 2429, 2025.

\bibitem{domingo2020mean}
C.~Domingo-Enrich, S.~Jelassi, A.~Mensch, G.~Rotskoff, and J.~Bruna.
\newblock A mean-field analysis of two-player zero-sum games.
\newblock {\em Advances in neural information processing systems},
  33:20215--20226, 2020.

\bibitem{glicksberg1952further}
I.~L. Glicksberg.
\newblock A further generalization of the {K}akutani fixed theorem, with
  application to {N}ash equilibrium points.
\newblock {\em Proceedings of the American Mathematical Society},
  3(1):170--174, 1952.

\bibitem{goodfellow2014generative}
I.~J. Goodfellow, J.~Pouget-Abadie, M.~Mirza, B.~Xu, D.~Warde-Farley, S.~Ozair,
  A.~Courville, and Y.~Bengio.
\newblock Generative adversarial nets.
\newblock {\em Advances in neural information processing systems}, 27, 2014.

\bibitem{grothaus2019weak}
M.~Grothaus and F.-Y. Wang.
\newblock Weak {P}oincar{\'e} inequalities for convergence rate of degenerate
  diffusion processes.
\newblock {\em The Annals of Probability}, 47(5):2930--2952, 2019.

\bibitem{hsieh2019finding}
Y.-P. Hsieh, C.~Liu, and V.~Cevher.
\newblock Finding mixed nash equilibria of generative adversarial networks.
\newblock In {\em International Conference on Machine Learning}, pages
  2810--2819. PMLR, 2019.

\bibitem{hu2021mean}
K.~Hu, Z.~Ren, D.~Šiška, and {\L}.~Szpruch.
\newblock {Mean-field Langevin dynamics and energy landscape of neural
  networks}.
\newblock {\em Annales de l'Institut Henri Poincaré, Probabilités et
  Statistiques}, 57(4):2043 -- 2065, 2021.

\bibitem{hu2019subexponential}
S.~Hu and X.~Wang.
\newblock Subexponential decay in kinetic {F}okker--{P}lanck equation: Weak
  hypocoercivity.
\newblock {\em Bernoulli}, 25(1):174--188, 2019.

\bibitem{isobe2025gradient}
N.~Isobe and S.~Shimoyama.
\newblock On gradient descent-ascent flows in metric spaces.
\newblock {\em arXiv preprint arXiv:2506.20258}, 2025.

\bibitem{jordan1998variational}
R.~Jordan, D.~Kinderlehrer, and F.~Otto.
\newblock The variational formulation of the {F}okker--{P}lanck equation.
\newblock {\em SIAM journal on mathematical analysis}, 29(1):1--17, 1998.

\bibitem{kavian2021fokker}
O.~Kavian, S.~Mischler, and M.~Ndao.
\newblock The {Fokker-Planck} equation with subcritical confinement force.
\newblock {\em Journal de Math{\'e}matiques Pures et Appliqu{\'e}es},
  151:171--211, 2021.

\bibitem{kim2024symmetric}
J.~Kim, K.~Yamamoto, K.~Oko, Z.~Yang, and T.~Suzuki.
\newblock Symmetric mean-field {L}angevin dynamics for distributional minimax
  problems.
\newblock In {\em International Conference on Learning Representations}, volume
  2024, pages 55056--55098, 2024.

\bibitem{lascu2024fisher}
R.-A. Lascu, M.~B. Majka, and {\L}.~Szpruch.
\newblock A {Fisher-Rao} gradient flow for entropic mean-field min-max games.
\newblock {\em Transactions on Machine Learning Research}, 2024.

\bibitem{lascu2025entropic}
R.-A. Lascu, M.~B. Majka, and {\L}.~Szpruch.
\newblock Entropic mean-field min--max problems via best response flow.
\newblock {\em Applied Mathematics \& Optimization}, 91(2):48, 2025.

\bibitem{liu2020poincare}
Y.~Liu.
\newblock {The Poincaré inequality and quadratic transportation-variance
  inequalities}.
\newblock {\em Electronic Journal of Probability}, 25(none):1 -- 16, 2020.

\bibitem{liu2025convergence}
Z.~Liu, F.~Liu, R.~Gao, and S.~Li.
\newblock Convergence of mean-field {L}angevin stochastic descent-ascent for
  distributional minimax optimization.
\newblock In {\em Forty-second International Conference on Machine Learning},
  2025.

\bibitem{lu2023two}
Y.~Lu.
\newblock Two-scale gradient descent ascent dynamics finds mixed {N}ash
  equilibria of continuous games: A mean-field perspective.
\newblock In {\em International Conference on Machine Learning}, pages
  22790--22811. PMLR, 2023.

\bibitem{lu2025convergence}
Y.~Lu and P.~Monmarch{\'e}.
\newblock Convergence of time-averaged mean field gradient descent dynamics for
  continuous multi-player zero-sum games.
\newblock {\em arXiv preprint arXiv:2505.07642}, 2025.

\bibitem{ma2022provably}
C.~Ma and L.~Ying.
\newblock Provably convergent quasistatic dynamics for mean-field two-player
  zero-sum games.
\newblock In {\em International Conference on Learning Representations}, 2022.

\bibitem{madry2017towards}
A.~Madry, A.~Makelov, L.~Schmidt, D.~Tsipras, and A.~Vladu.
\newblock Towards deep learning models resistant to adversarial attacks.
\newblock In {\em International Conference on Learning Representations}, 2018.

\bibitem{mei2018mean}
S.~Mei, A.~Montanari, and P.-M. Nguyen.
\newblock A mean field view of the landscape of two-layer neural networks.
\newblock {\em Proceedings of the National Academy of Sciences},
  115(33):E7665--E7671, 2018.

\bibitem{monmarche2025local}
P.~Monmarch{\'e} and J.~Reygner.
\newblock Local convergence rates for {Wasserstein gradient flows and
  McKean-Vlasov} equations with multiple stationary solutions.
\newblock {\em Probability Theory and Related Fields}, pages 1--59, 2025.

\bibitem{mourrat2026pl}
J.-C. Mourrat.
\newblock {PL} conditions do not guarantee convergence of gradient
  descent-ascent dynamics.
\newblock {\em arXiv preprint arXiv:2602.16517}, 2026.

\bibitem{mourrat2026oscillating}
J.-C. Mourrat and L.~Pillaud-Vivien.
\newblock Oscillating solutions to the mean-field {L}angevin descent-ascent
  flow.
\newblock {\em arXiv preprint arXiv:2604.11134}, 2026.

\bibitem{Muckenhoupt1972}
B.~Muckenhoupt.
\newblock Hardy's inequality with weights.
\newblock {\em Studia Mathematica}, 44(1):31--38, 1972.

\bibitem{nitanda2022convex}
A.~Nitanda, D.~Wu, and T.~Suzuki.
\newblock Convex analysis of the mean field {L}angevin dynamics.
\newblock In G.~Camps-Valls, F.~J.~R. Ruiz, and I.~Valera, editors, {\em
  Proceedings of The 25th International Conference on Artificial Intelligence
  and Statistics}, volume 151 of {\em Proceedings of Machine Learning
  Research}, pages 9741--9757. PMLR, 28--30 Mar 2022.

\bibitem{peyre2018comparison}
R.~Peyre.
\newblock Comparison between {$W^2$} distance and {$\dot{H}^{-1}$} norm, and
  localization of {W}asserstein distance.
\newblock {\em ESAIM: Control, Optimisation and Calculus of Variations},
  24(4):1489--1501, 2018.

\bibitem{rockner2001weak}
M.~R{\"o}ckner and F.-Y. Wang.
\newblock Weak {P}oincar{\'e} inequalities and ${L^2}$-convergence rates of
  {M}arkov semigroups.
\newblock {\em Journal of Functional Analysis}, 185(2):564--603, 2001.

\bibitem{seo2026local}
G.~Seo, M.~Shin, P.~Monmarch{\'e}, and B.~Choi.
\newblock Local exponential stability of mean-field {L}angevin descent-ascent
  in {W}asserstein space.
\newblock {\em arXiv preprint arXiv:2602.01564}, 2026.

\bibitem{sinha2017certifying}
A.~Sinha, H.~Namkoong, R.~Volpi, and J.~Duchi.
\newblock Certifying some distributional robustness with principled adversarial
  training.
\newblock In {\em International Conference on Learning Representations}, 2018.

\bibitem{suzuki2023uniform}
T.~Suzuki, A.~Nitanda, and D.~Wu.
\newblock Uniform-in-time propagation of chaos for the mean-field gradient
  {L}angevin dynamics.
\newblock In {\em International Conference on Learning Representations}, 2023.

\bibitem{suzuki2023convergence}
T.~Suzuki, D.~Wu, and A.~Nitanda.
\newblock Convergence of mean-field langevin dynamics: time-space
  discretization, stochastic gradient, and variance reduction.
\newblock In A.~Oh, T.~Naumann, A.~Globerson, K.~Saenko, M.~Hardt, and
  S.~Levine, editors, {\em Advances in Neural Information Processing Systems},
  volume~36, pages 15545--15577. Curran Associates, Inc., 2023.

\bibitem{wang2024open}
G.~Wang and L.~Chizat.
\newblock Open problem: Convergence of single-timescale mean-field langevin
  descent-ascent for two-player zero-sum games.
\newblock In {\em The Thirty Seventh Annual Conference on Learning Theory},
  pages 5345--5350. PMLR, 2024.

\bibitem{OJMO_2025__6__A1_0}
G.~Wang and L.~Chizat.
\newblock An exponentially converging particle method for the mixed {Nash}
  equilibrium of continuous games.
\newblock {\em Open Journal of Mathematical Optimization}, 6:1--66, 2025.

\bibitem{wang2026local}
G.~Wang and L.~Chizat.
\newblock Local convergence of mean-field {L}angevin dynamics: from gradient
  flows to linearly monotone games.
\newblock {\em arXiv preprint arXiv:2602.11999}, 2026.

\end{thebibliography}

\end{document}